\documentclass[11pt,reqno]{amsart}

\usepackage{color}
\usepackage{amsmath, amsthm, amssymb}
\usepackage{amsfonts}
\usepackage{enumitem}
\usepackage[ansinew]{inputenc}
\usepackage[dvips]{epsfig}
\usepackage{graphicx}
\usepackage[english]{babel}
\usepackage{hyperref}
\theoremstyle{plain}
\newtheorem{thm}{Theorem}[section]
\newtheorem{cor}[thm]{Corollary}
\newtheorem{lem}[thm]{Lemma}
\newtheorem{prop}[thm]{Proposition}

\theoremstyle{definition}

\theoremstyle{remark}
\newtheorem{rem}[thm]{Remark}

\numberwithin{equation}{section}

\newcommand{\de}{\partial}

\newcommand{\fls}{(-\Delta)^s}
\newcommand{\R}{\mathbb{R}}

\newcommand{\N}{\mathbb{N}}

\newcommand{\eps}{\varepsilon}
\newcommand{\ub}{\boldsymbol{u}}
\newcommand{\fb}{\boldsymbol{f}}
\newcommand{\nb}{\boldsymbol{n}}
\newcommand{\be}{\boldsymbol{e}}

\newcommand{\average}{{\mathchoice {\kern1ex\vcenter{\hrule height.4pt
width 6pt depth0pt} \kern-9.7pt} {\kern1ex\vcenter{\hrule
height.4pt width 4.3pt depth0pt} \kern-7pt} {} {} }}
\newcommand{\ave}{\average\int}
\def\R{\mathbb{R}}

\newcommand{\mres}{\mathbin{\vrule height 1.6ex depth 0pt width
0.13ex\vrule height 0.13ex depth 0pt width 1.3ex}}

\begin{document}

\title[The Thin Obstacle Problem: A Survey]{The Thin Obstacle Problem: A Survey}

\author{Xavier Fern\'andez-Real}
\address{EPFL SB, Station 8, CH-1015 Lausanne, Switzerland}
\email{xavier.fernandez-real@epfl.ch}



\thanks{This work has received funding from the European Research Council (ERC) under the Grant Agreement No 721675, and from the Swiss National Science Foundation under the SNF Grant 200021\_182565}

\setcounter{tocdepth}{1}

\begin{abstract}
In this work we present a general introduction to the Signorini problem (or thin obstacle problem). It is a self-contained survey that aims to cover the main currently known results regarding the thin obstacle problem. We present the theory with some proofs: from the optimal regularity of solutions and classification of free boundary points, to more recent results on the non-regular part of the free boundary and generic regularity. 
\end{abstract}

\maketitle
\tableofcontents

\section{Introduction}

The Signorini problem (also known as the thin or boundary obstacle problem) is a classical free boundary problem that was originally studied by Antonio Signorini in connection with linear elasticity \cite{Sig33, Sig59}. The problem was originally named by Signorini himself \emph{problem with ambiguous boundary conditions}, in the sense that the solution of the problem at each boundary point must satisfy one of two different possible boundary conditions, and it is not known a priori which point satisfies which condition. 

Whereas the original problem involved a system of equations, its scalar version gained further attention in the seventies due to its connection to many other areas, which then led to it being widely studied by the mathematical community.
In fact, in addition to elasticity, lower dimensional obstacle problems also appear in describing osmosis through semi-permeable membranes as well as boundary heat control (see, e.g., \cite{DL76}).
Moreover, they often are local formulations of fractional obstacle problems, another important class of obstacle problems.
Fractional obstacle problems can be found in the optimal stopping problem for L\'evy processes, and can be used to model American option prices (see \cite{Mer76,CT04}). 
They also appear in the study of anomalous diffusion, \cite{BG90}, the study of quasi-geostrophic flows, \cite{CV10}, and in studies of the interaction energy of probability measures under singular potentials, \cite{CDM16}. 
(We refer to \cite{Ros18} for an extensive bibliography on the applications of obstacle-type problems.) Finally, thin obstacle problems also appear when studying the classical obstacle problem, \cite{FS17, FRS19}.

The aim of this survey is to offer the interested reader a short, self-contained introduction to the thin obstacle problem and, in particular, to the structure of the free boundary, involving some recent results in this context. Other very interesting expository works in the field include \cite{Sal12, DaSa18}, together with the book \cite[Chapter 9]{PSU12}, which give in full detail the proofs of approach existence and optimal regularity of solutions, as well as a study of regular and singular free boundary points. 

We believe that the current work complements well the previous expository works \cite{PSU12, Sal12, DaSa18}. We present a different (and shorter) proof of the classification of free boundary points and optimal regularity based on the ideas of \cite{CRS17}-\cite{RS17}: traditionally, optimal regularity is shown by means of a monotonicity formula, and then, together with Almgren's monotonicity formula, the classification of free boundary points can be established. Here, we skip the first step and are able to perform the classification directly, without knowing the optimal regularity, and deduce a posteriori that the lowest possible homogeneity is $\frac32$. On the other hand, we also give an overview on more recent results for the thin obstacle problem (\cite{BFR18, FS18,  FR19, FJ18, CSV19, FRS19}).

\subsection{Structure of the survey} This work is organized as follows.

We start with a short motivation for the thin obstacle problem in Section~\ref{sec.2} coming from electrostatics. In Section~\ref{sec.3}  we introduce, without giving proofs, the thin obstacle problem and the basic regularity results for the solution. We refer the interested reader to \cite{Caf79, PSU12, Sal12, DaSa18} for the proofs of these results. 

We then move to Sections~\ref{sec.6} and \ref{sec.7} on the (first) classification of free boundary points and the study of regular points. In these sections we present some simple proofs which are different from \cite{AC04, ACS08, PSU12, Sal12, DaSa18}.

We then move to Section~\ref{sec.8} in which we study singular points. In order to keep the exposition simple, we assume the obstacle to be zero (alternatively, analytic). We introduce the singular set for the thin obstacle problem, we discuss a non-degenerate case in which regular and singular points form the whole of the free boundary, we introduce higher order regularity results for the singular set, and we briefly mention epiperimetric inequalities for the thin obstacle problem and their use in showing regularity of the singular set. Then, in Section~\ref{sec.9} we introduce the rest of the free boundary, known as the set of \emph{Other points}. We here present the known results for each of the components of this set, and give a self-contained proof of the smallness of one of them. 

Then, in Section~\ref{sec.cinf} we briefly discuss the case of $C^\infty$ obstacles and the presence, in this case, of a different type of free boundary point with infinite frequency; and in Section~\ref{sec.11} we introduce the generic free boundary regularity. We state the recent results in this direction, that show that even if the free boundary could have a bad structure for certain configurations, such situations are \emph{not common}. We finish with a short summary of the structure of the free boundary in Section~\ref{sec.12}.

\section{A problem from elastostatics}
\label{sec.2}
Consider an elastic body $\Omega\subset \R^3$, anisotropic and non-homogeneous, in an equilibrium configuration, that must remain on one side of a frictionless surface. Let us denote $\ub = (u^1, u^2, u^3) : \Omega \to \R^3$ the displacement vector of the elastic body, $\Omega$, constrained to be on one side of a surface $\Pi$ (in particular, the elastic body moves from the $\Omega$ configuration to $\Omega+\ub(\Omega)$). We divide the boundary into $\de\Omega = \Sigma_D\cup \Sigma_S$. The body is free (or clamped, $\ub \equiv 0$) at $\Sigma_D$, whereas $\Sigma_S$ represents the part of the boundary subject to the constraint, that is, $\Sigma_S = \partial \Omega\cap \Pi$. Alternatively, one can interpret $\Sigma_S$ itself as the frictionless surface that is constraining the body $\Omega$, understanding that only a subset of $\Sigma_S$ is actually exerting the constraint on the displacement. This will be more clear below. 

Let us assume small displacements, so that we can consider the linearized strain tensor 
\[
\eps_{ij}(\ub) = \frac12 (u^i_{x_j} + u^j_{x_i}),\quad 1\le i, j \le 3. 
\]
Considering an elastic potential energy of the form $W(\eps) = a_{ijkh}(x) \eps_{ij}\eps_{kh}$, for some functions $a_{ijhk}(x) \in C^\infty(\overline{\Omega})$ (where, from now on, we are using the Einstein notation of implicit summation over repeated indices), then the stress tensor has the form 
\[
\sigma_{ij}(\ub) = a_{ijhk}(x) \eps_{hk}(\ub).
\] 
We also impose that $a_{ijhk}$ are elliptic and with symmetry conditions 
\[
\begin{split}
& a_{ijhk}(x) \zeta_{ij}\zeta_{hk} \ge \lambda |\zeta|^2\quad \textrm{for all }\zeta\in \R^{n\times n}\textrm{ such that }\zeta_{ij} = \zeta_{ji},\\
& a_{ijhk}(x) = a_{jihk}(x) = a_{ijkh}(x),\quad\textrm{ for }x\in \Omega. 
\end{split}
\]
Let us also assume that $\Omega$ is subject to the body forces $\fb = (f^1, f^2, f^3)$, so that by the general equilibrium equations we have 
\[
\frac{\de\sigma_{ij}(\ub)}{\de x_j} = f^i,\quad \textrm{in }\Omega, \quad \textrm{for } i = 1, 2, 3.
\]
From the definitions of $\sigma(\ub)$ and $\eps_{ij}(\ub)$ above, this is a second order system, and from the definition of $a_{ijhk}$, it is elliptic. Thus, the displacement vector satisfies an elliptic second order linear system inside $\Omega$. We just need to impose boundary conditions on $\Sigma_S$ (the conditions on $\Sigma_D$ are given by the problem, we can think of $\ub \equiv 0$ there).  

Let us denote by $\nb$ the outward unit normal vector to $x\in \de\Omega$. Notice that, by assumption, the stresses in the normal direction $\nb$ on $\Sigma_S$, $\sigma_{ij}(\ub) \nb_i$, must be compressive in the normal direction, and zero in the tangential direction (due to the frictionless surface). That is, 
\begin{align}
\label{eq.sigmann}\sigma_{ij}(\ub)\nb_i\nb_j &\le 0 \quad\textrm{on $\Sigma_S$,} \\
\nonumber \sigma_{ij}(\ub)\nb_i\tau_j &=  0 \quad\textrm{on $\Sigma_S$ and for all $\tau\in \R^n$ with $\tau\cdot\nb = 0$. }
\end{align}
On the other hand, we have the kinematical contact condition, encoding the fact that there exists a surface exerting a constraint and the body cannot cross it (under small displacements, or assuming simply that $\Pi$ is a plane):
\begin{equation}
\label{eq.ub0}
\ub \cdot\nb \le 0,\quad\textrm{ on }\Sigma_S. 
\end{equation}

In fact, conditions \eqref{eq.sigmann}-\eqref{eq.ub0} are complimentary, in the sense that 
\begin{equation}
\label{eq.ub1}
(\ub \cdot\nb) \cdot (\sigma_{ij}(\ub)\nb_i\nb_j) = 0 \quad\textrm{ on $\Sigma_S$,}
\end{equation}
and we are dividing $\Sigma_S$ into two regions: those where the body separates from $\Pi$ and those where it remains touching $\Pi$.
That is, if there is an active normal stress at a point $x\in \Sigma_S$, $\sigma_{ij}(\ub(x))\nb_i(x)\nb_j(x) < 0$, then it means that the elastic body is being constrained by $\Sigma_S$ (or $\Pi$) at $x$, and thus we are in the contact area and there is no normal displacement, $\ub(x) \cdot\nb(x) = 0$. Alternatively, if there is a normal displacement, $\ub(x) \cdot\nb(x) < 0$, it means that there is no active obstacle and thus no normal stress, $\sigma_{ij}(\ub(x))\nb_i(x)\nb_j(x) = 0$. This is precisely what \emph{ambiguous boundary condition means}: 

For each $x\in \Sigma_S$ we have that one of the following two conditions holds
\begin{equation}
\label{eq.ambiguous}
\textrm{either}\quad \left\{\begin{array}{rcl}
\sigma_{ij}(\ub(x))\nb_i(x)\nb_j(x) &\le & 0\\
\ub(x) \cdot\nb(x) &= &0,
\end{array}
\right.
\quad \textrm{or}\quad \left\{\begin{array}{rcl}
\sigma_{ij}(\ub(x))\nb_i(x)\nb_j(x) &= & 0\\
\ub(x) \cdot\nb(x) &< &0,
\end{array}
\right.
\end{equation}
and a priori, we do not know which of the condition is being fulfilled at each point. The Signorini problem is a \emph{free boundary problem} because the set $\Sigma_S$ can be divided into two different sets according to which of the conditions \eqref{eq.ambiguous} holds, and these sets are, a priori, unknown. The boundary between both sets is what is known as the \emph{free boundary}. 

The previous is a strong formulation of the Signorini problem, which assumed a priori that all solutions and data are smooth. In order to prove existence and uniqueness, however, one requires the use of variational inequalities with (convex) constraints in the set of admissible functions. 

The first one to approach the existence and uniqueness from a variational point of view was Fichera in \cite{Fic64}. We also refer to the work of Lions and Stampacchia \cite{LS67}, where a general theory of variational inequalities was developed, which later led to the scalar version of the Signorini problem, and its interpretation as a minimization problem with admissible functions constrained to be above zero on certain fixed closed sets. Later, in \cite{DL76},  Duvaut and Lions studied the problem and its applications to mechanics and physics. 

Finally, we refer to \cite{Kin81, KO88} for more details into the strong and weak formulation of the vectorial Signorini problem and its properties, and to the more recent \cite{And16} for the optimal regularity of the solution and regularity of the free boundary.

\section{The thin obstacle problem}
\label{sec.3}

In this work we will focus our attention to the scalar version of the Signorini problem from elasticity: our function, $u$, would correspond to an appropriate limit in the normal components of the displacement vector, $\boldsymbol{u}_n$. Our obstacle, $\varphi$, adds generality to the problem, and would correspond to the possible displacement of the frictionless surface $\partial\Omega$ while performing $\boldsymbol{u}$. (We refer the interested reader to \cite[Example 1.5]{CDV19} for a deduction of this fact.) As explained above, this problem also appears in biology, physics, and even finance. Thus, from now on, functions are scalar.

Let us denote \[
x = (x', x_{n+1}) \in \R^n\times\R,\quad B_1 := \{x\in \R^{n+1} : |x|< 1\},\quad\text{and}\quad B_1^+ := B_1\cap \{x_{n+1} > 0\}.
\] We say that $u :\overline{B_1^+}\to \R$ is a solution to the Signorini problem or thin obstacle problem with  smooth obstacle $\varphi $ defined on $B_1' := B_1\cap \{x_{n+1} = 0\}$, and with smooth boundary data $g$ on $\de B_1 \cap \{x_{n+1} > 0\}$, if $u$ solves
\begin{equation}
\label{eq.thinobst_intro}
  \left\{ \begin{array}{rcll}
  \Delta u&=&0 & \textrm{ in } B_1^+\\
    u&=&g & \textrm{ on } \de B_1 \cap \{x_{n+1} > 0\}\\
  \de_{x_{n+1}} u\cdot( u-\varphi) & = &0 & \textrm{ on } B_1\cap \{x_{{n+1}} = 0\}\\
    -\de_{x_{n+1}} u & \ge &0 & \textrm{ on } B_1\cap \{x_{{n+1}} = 0\}\\
       u-\varphi & \ge &0 & \textrm{ on } B_1\cap \{x_{{n+1}} = 0\},
  \end{array}\right.
\end{equation}
where we are also assuming that the compatibility condition $g\ge \varphi$ on $\de B_1\cap \{x_{n+1} = 0\}$ holds. Notice the analogy with the ambiguous compatibility conditions \eqref{eq.sigmann}-\eqref{eq.ub0}-\eqref{eq.ub1} or \eqref{eq.ambiguous}: the set with Dirichlet conditions, $\Sigma_D$ above, is $\de B_1\cap \{x_{n+1}> 0\}$, where $u = g$ is imposed; whereas the set with ambiguous boundary conditions, $\Sigma_S$ above, is now $B_1'$. That is, at each point $x = (x',0 )\in B_1'$ we have that 
\[
\textrm{either}\quad \left\{\begin{array}{rcl}
-\de_{x_{n+1}} u(x)&\ge & 0\\
u(x)-\varphi(x')  &= &0,
\end{array}
\right.
\quad \textrm{or}\quad \left\{\begin{array}{rcl}
-\de_{x_{n+1}} u(x)&= & 0\\
u(x)-\varphi(x')  &> &0.
\end{array}
\right.
\]

An alternative way to write the ambiguous boundary conditions in \eqref{eq.thinobst_intro} is by imposing a nonlinear condition on $B_1'$ involving $u$ and $\de_{x_{n+1}}u $ as 
\begin{equation}
\label{eq.thinobst_intro_2}
  \left\{ \begin{array}{rcll}
  \Delta u&=&0 & \textrm{ in } B_1^+\\
  \min\{-\de_{x_{n+1}} u,  u-\varphi\} & = &0 & \textrm{ on } B_1\cap \{x_{{n+1}} = 0\},
  \end{array}\right.
\end{equation}
with $u = g$ on $\de B_1 \cap \{x_{n+1} > 0\}$. This is the strong formulation of the Signorini problem. 

In order to prove existence (and uniqueness) of solutions, we need to study the weak formulation of the problem: a priori, we do not know any regularity for the solution. 

Consider a bounded domain $\Omega\subset \R^n$, and a closed set $\mathcal{C}\subset \Omega$. Let, also, $\phi: C(\mathcal{C})\to \R$ be a continuous function. In \cite{LS67}, Lions and Stampacchia prove the existence and uniqueness of a solution to the variational problem
\begin{equation}
\label{eq.min}
\min_{v\in \mathcal{K}}\int_\Omega |\nabla v|^2
\end{equation}
where $\mathcal{K} = \{v\in H_0^1(\Omega) : v\ge \phi\textrm{ on }\mathcal{C}\}$. Moreover, they also show that such solution is the smallest supersolution. 

If $\mathcal{C} = \overline{\Omega}$, \eqref{eq.min} is also known as the (classical) obstacle problem: finding the function with smallest Dirichlet energy among all those which lie above a fixed obstacle $\phi$. This problem has been thoroughly studied in the last fifty years (see 	\cite{LS67, KN77, Caf77, CR77, Caf98, Wei99, PSU12} and references therein), and we will sometimes refer to it also as the \emph{thick} or \emph{classical} obstacle problem.

Our problem, \eqref{eq.thinobst_intro_2}, corresponds to the case when $\mathcal{C}$ is lower dimensional, with codimension 1. Notice that simple capacity arguments yield that, if $\mathcal{C}$ has codimension 2 or higher, then the restriction of functions in $H^1_0$ to $\mathcal{C}$ does not have any effect on the minimization of the Dirichlet energy, and thus we would simply be solving the classical Laplace equation. This means that, in higher codimension, there is in general no minimizer. 

Thus, \eqref{eq.thinobst_intro_2} are the Euler--Lagrange equations of the following variational problem
\begin{equation}
\label{eq.min1}
\min_{v\in  \mathcal{K}^*}\int_{B_1^+} |\nabla v |^2,
\end{equation}
where 
\[
\mathcal{K}^* = \{v\in H^1(B_1^+) : v = g \textrm{ on }\de B_1\cap \{x_{n+1}  > 0\}, v \ge \varphi\textrm{ on }B_1\cap \{x_{n+1}  = 0\}\}.
\]
Notice that the expressions $v = g$ on $\de B_1\cap \{x_{n+1}  > 0\}$ and $v\ge \varphi$ on $B_1\cap \{x_{n+1}  = 0\}$ must be understood in the trace sense. 
 The existence and uniqueness of a solution, as in \cite{LS67}, follows by classical methods: take a minimizing sequence, and by lower semicontinuity of the Dirichlet energy, and the compactness of the trace embeddings into $H^1$, the limit is also an admissible function. The uniqueness follows by strict convexity of the functional. 
 
In some cases, the thin obstacle problem is posed in the whole ball $B_1$, and thus we consider  
 \begin{equation}
\label{eq.min2}
\min_{v\in  \mathcal{K}^{**}}\int_{B_1} |\nabla v |^2,\quad \mathcal{K}^{**} = \{v\in H^1(B_1) : v = g \textrm{ on }\de B_1,v \ge \varphi\textrm{ on }B_1\cap \{x_{n+1}  = 0\}\},
\end{equation}
for some function $g\in C(\de B_1)$. In this case, the Euler--Lagrange equations are formally
\begin{equation}
\label{eq.thinobst_intro_3}
  \left\{ \begin{array}{rcll}
  u & \ge & \varphi & \textrm{ on } B_1\cap \{x_{{n+1}} = 0\}\\
  \Delta u&=&0 & \textrm{ in } B_1\setminus\left(\{x_{n+1} = 0\}\cap \{u = \varphi\}\right)\\
  \Delta u & \le  &0 & \textrm{ in } B_1,
  \end{array}\right.
\end{equation}
with the added condition that $u = g$ on $\de B_1$. Alternatively, making the parallelism with \eqref{eq.thinobst_intro_2}, one could formally write 
\begin{equation}
  \left\{ \begin{array}{rcll}
  \Delta u&=&0 & \textrm{ in } B_1\setminus\{x_{n+1} = 0\}\\
  \min\{-\Delta u,  u-\varphi\} & = &0 & \textrm{ on } B_1\cap \{x_{{n+1}} = 0\},
  \end{array}\right.
\end{equation}
understanding that $\Delta u$ is defined only in the distributional sense. 

\begin{rem}
\label{rem.geven}
Notice that if $g$ is even with respect to $x_{n+1}$, the solution to \eqref{eq.thinobst_intro_3} is even as well, and we  recover a problem of the form \eqref{eq.thinobst_intro_2}. On the other hand, for general $g$, one can study the symmetrised function $\bar u(x', x_{n+1}) = \frac12\left( u(x', x_{n+1})+  u(x', -x_{n+1})\right)$, which has the same regularity and contact set as $u$. Thus, in order to study \eqref{eq.thinobst_intro_3} one can always assume that $u$ is even in $x_{n+1}$, and this  is enough to study \eqref{eq.thinobst_intro_2}. 
\end{rem}

Notice, also, that in \eqref{eq.thinobst_intro_3} the condition $\Delta u \le 0$ needs to be understood in the sense of distributions. In fact, $\Delta u$ is a (non-positive) measure concentrated on $\{u = 0\}$. We can explicitly compute it by taking any test function $\phi \in C^\infty_c(B_1)$ even in $x_{n+1}$,
\begin{align*}
- \langle \Delta u, \phi\rangle & = 2\int_{B_1^+} \nabla u\cdot\nabla \phi = 2\lim_{\eps\downarrow 0}\int_{B_1^+\cap\{x_{n+1} \ge \eps\}} \nabla u \cdot\nabla\phi\\
& = -2\lim_{\eps\downarrow 0}\int_{B_1^+\cap\{x_{n+1} = \eps\}} \de_{x_{n+1}}u \, \phi = -2\int_{B_1\cap\{x_{n+1} = 0\}}\de_{x_{n+1}}^+ u \, \phi.
\end{align*}

That is, 
\begin{equation}
\label{eq.delta_xn}
\Delta u = 2\de_{x_{n+1}}^+ u\, \mathcal{H}^n\mres \left(B_1\cap\{x_{n+1} = 0\}\right),
\end{equation}
where $\de_{x_{n+1}}^+ u = \lim_{\eps\downarrow 0} \de_{x_{n+1}} u(x', \eps)$. 

\begin{rem}
In the derivation of \eqref{eq.delta_xn}, apart from \eqref{eq.thinobst_intro_3}, we have also used integrability of $\nabla u$, and that the trace of the normal derivative is well-defined. This follows because, in fact, as we will show later, the solution to the thin obstacle problem is Lipschitz, and is continuously differentiable up to the obstacle.  
\end{rem}

\begin{rem}
Problem \eqref{eq.min2} can be seen as a first order approximation of the Plateau problem with a lower dimensional obstacle, originally introduced by De Giorgi \cite{DGthin}, which has also been studied in the last years \cite{Deacu, FS, FeSe18}. Indeed, the Dirichlet functional corresponds to the area functional (up to a constant) for flat graphs.
\end{rem}

Finally, let us end this section by mentioning other possible constructions of solutions. As mentioned above, the solution to the previous minimization problem can also be recovered as the least supersolution. That is, the minimizer $u$ to \eqref{eq.min2} equals to the pointwise infimum
\[
\begin{split}
u(x) = \inf\big\{ v(x) : v\in C^2(B_1), -\Delta v \ge 0\textrm{ in }B_1, v \ge \varphi \textrm{ on }B_1 \cap & \{x_{n+1} = 0\},  v \ge g \textrm{ on }\de B_1\big\},
\end{split}
\]
the least supersolution above the thin obstacle. The fact that such function satisfies \eqref{eq.thinobst_intro_3} can be proved by means of Perron's method, analogously to the Laplace equation. 

As a final characterization of the construction of the solution, we refer to penalization arguments. In this case there are two ways to penalize:

On the one hand, we can \emph{expand} the obstacle, and work with the classical obstacle problem. That is, we can consider as obstacle $\varphi_\eps(x) = \varphi(x') - \eps^{-1}x_{n+1}^2$ with $\eps > 0$ very small, which is now defined in the whole domain $B_1$. Then, the solutions to the thick obstacle problem with increasingly thinner obstacles $\varphi_\eps$ (letting $\eps\downarrow 0$), converging to our thin obstacle, will converge to the solution to our problem. Alternatively, we can even avoid the penalization step: the solutions to the thin obstacle problem must coincide with the solution of the thick obstacle problem, with obstacle $\bar\varphi: B_1^+\to \R$ given by the solution to $\Delta \bar\varphi = 0$ in $B_1^+$, $\bar\varphi = \varphi$ on $B_1\cap \{x_{n+1} = 0\}$, $\bar\varphi = g$ on $\de B_1\cap \{x_{n+1}> 0\}$ (which can have a hard wedge  on $\{x_{n+1} = 0\}$). Notice that $\bar\varphi$ itself is not the solution to the thin obstacle problem since, a priori, it is not a supersolution across $\{x_{n+1} = 0\}$. 

On the other hand, we can penalize \eqref{eq.thinobst_intro_2} by replacing the ambiguous boundary condition on $\{x_{n+1} = 0\}$, by considering solutions $u^\eps$ with the Neumann boundary condition $u^\eps_{x_{n+1}}  = \eps^{-1}\min\{0, u-\varphi\}$ on $\{x_{n+1} = 0\}$. By letting $\eps\downarrow 0$, $u^\eps$ converges to a solution to our problem. 

\subsection{Relation with the fractional obstacle problem}

Let us consider the thin obstacle problem \eqref{eq.thinobst_intro_2} posed in the whole $\R^{n+1}$, for some smooth obstacle $\varphi:\R^n\to\R$ with compact support. That is,  we denote $\R_+^{n+1} = \R^{n+1}\cap\{x_{n+1} > 0\}$ and consider a solution to 
\begin{equation}
\label{eq.top_wd}
  \left\{ \begin{array}{rcll}
  \Delta u&=&0 & \textrm{ in } \R^{n+1}_+\\
   u(x',0) & \ge &\varphi(x') & \textrm{ for } x'\in\R^n\\
   \de_{x_{n+1}} u(x',0) & = &0 & \textrm{ if } u(x', 0) > \varphi(x')\\
   \de_{x_{n+1}} u(x',0) & \le &0 & \textrm{ if } u(x', 0) = \varphi(x')\\
   u(x)&\to&0 & \textrm{ as } |x|\to\infty.
  \end{array}\right.
\end{equation}

If we denote by $\bar u: \R^n\to\R$ the restriction of $u$ to $\{x_{n+1} = 0\}$, then we can simply reformulate the problem in terms of $\bar u$ instead of $u$, given that $u$ is just the harmonic extension (vanishing at infinity) of $\bar u$ to $\R^{n+1}_+$. That is, by means of the Poisson kernel in the half-space, 
\[
u(x', x_{n+1}) = \left[P(x_{n+1}, \cdot)*u\right](x')= c_n\int_{\R^n} \frac{x_{n+1} \bar u(y')\,dy'}{(x_{n+1}^2+|x'-y'|^2)^{\frac{n+1}{2}}}
\]
for some dimensional constant $c_n$. Thus, after a careful computation and taking limits $x_{n+1}\downarrow 0$, one obtains
\[
-\de_{x_{n+1}} u(x', 0) = c_n{\rm PV}\int_{\R^n}\frac{\bar u(x') - \bar u(y')}{|x'-y'|^{n+1}}\,dy' =: (-\Delta)^{\frac12}\bar u (x'),
\]
where the integral needs to be understood in the principal value sense. We have introduced here an integro-differential operator, acting on $\bar u$, $(-\Delta)^{\frac12}$, known as the fractional Laplacian of order 1 (in the sense that it is $1$-homogeneous, $(-\Delta)^{\frac12}(\bar v(r\cdot)) = r((-\Delta)^{\frac12}\bar v)(r\cdot)$).

Let us very briefly justify the choice of notation $(-\Delta)^{\frac12}$ in terms of the discussion above. Given a smooth (say, $C^2$) function $\bar u$, $(-\Delta)^{\frac12}\bar u$ is the normal derivative of its harmonic extension. If one repeats this procedure, and takes the harmonic extension of $(-\Delta)^{\frac12}\bar u$, it is simply $\de_{x_{n+1}} u$. Thus, $(-\Delta)^{\frac12}(-\Delta)^{\frac12} \bar u = \de^2_{x_{n+1}} u = -\Delta_{x'}\bar u$, where we are using the fact that $\Delta u = 0$ (up to the boundary), and we denote $\Delta = \Delta_x' + \de^2_{x_{n+1}}$. 

In all, problem \eqref{eq.top_wd} can be rewritten in terms of $\bar u$ as 
\begin{equation}
\label{eq.top_f12}
  \left\{ \begin{array}{rcll}
   \bar u & \ge &\varphi & \textrm{ in } \R^n\\
   (-\Delta)^{\frac12} \bar u & = &0 & \textrm{ if } u > \varphi\\
   (-\Delta)^{\frac12} \bar u & \ge &0 & \textrm{ if } u = \varphi\\
   \bar u(x')&\to&0 & \textrm{ as } |x'|\to\infty,
  \end{array}\right.
\end{equation}
which is the formulation of the classical (or thick) global obstacle problem, with obstacle $\varphi$ and operator $(-\Delta)^{\frac12}$, also referred to as \emph{fractional obstacle problem}. Notice that now we are considering a function $\bar u$ that remains above the obstacle $\varphi$ in the whole domain (compared to before, where we only needed this condition imposed on a lower dimensional manifold). 

Similarly, one can consider the fractional obstacle problem in a bounded domain $\Omega\subset\R^n$ with a (smooth) obstacle $\varphi:\Omega\to\R$ by imposing exterior boundary conditions with sufficient decay, $\bar g: \R^n\setminus\Omega\to \R$,
\begin{equation}
\label{eq.top_f12_Om}
  \left\{ \begin{array}{rcll}
   \bar u & \ge &\varphi & \textrm{ in } \Omega\\
   (-\Delta)^{\frac12} \bar u & = &0 & \textrm{ in } \Omega\cap\{u > \varphi\}\\
   (-\Delta)^{\frac12} \bar u & \ge &0 & \textrm{ in } \Omega\cap\{u = \varphi\}\\
   \bar u&= &\bar g & \textrm{ in } \R^n\setminus\Omega.
  \end{array}\right.
\end{equation}
Thus, in order to study the solution to \eqref{eq.top_f12_Om}, by taking its harmonic extension $\bar u$, it is enough to study the solutions to \eqref{eq.thinobst_intro_2}.

Finally, another characterization of the fractional Laplacian, $(-\Delta)^{\frac12}$, is via Fourier transforms. In this way, one can also characterize (up to a constant) general fractional Laplacians of order $2s$, with $0 < s < 1$, as 
\[
\mathcal{F}(\fls\bar u) (\xi) = |\xi|^{2s}\mathcal{F}(\bar u)(\xi),
\]
where $\mathcal{F}$ denotes the Fourier transform. The operator, which now has order $2s$, can be explicitly written as 
\[
\fls \bar u (x') = c_{n,s}{\rm PV}\int_{\R^n}\frac{\bar u(x') - \bar u(y')}{|x'-y'|^{n+2s}}\, dy'.
\]

In this way, one can consider general obstacle problems with nonlocal operator $\mathcal{L} = \fls$
\begin{equation}
\label{eq.top_L_Om}
  \left\{ \begin{array}{rcll}
   \bar u & \ge &\varphi & \textrm{ in } \Omega\\
   \mathcal{L}\bar u & = &0 & \textrm{ in } \Omega\cap\{u > \varphi\}\\
   \mathcal{L} \bar u & \ge &0 & \textrm{ in } \Omega\cap\{u = \varphi\}\\
   \bar u&= &\bar g & \textrm{ in } \R^n\setminus\Omega.
  \end{array}\right.
\end{equation}
(See, e.g., \cite{Sil07}.) As we have seen, the fractional Laplacian $(-\Delta)^{\frac12}$ can be recovered as the normal derivative of the harmonic extension towards one extra dimension (cf. \eqref{eq.top_f12_Om}-\eqref{eq.thinobst_intro_2}). Caffarelli and Silvestre showed in \cite{CS07} that the fractional Laplacian of order $\fls$ can also be recovered by extending through suitable operators. Thus, if one considers the operator 
\[
L_a u := {\rm div}(|x_{n+1}|^a\nabla u),\qquad a = 1-2s \in (-1,1),
\] 
then the even $a$-harmonic extension of the solution $\bar u$ to \eqref{eq.top_L_Om} (that is, $u$ with $L_a u = 0$ in $x_{n+1} > 0$ and $u(x', x_{n+1}) = u(x', -x_{n+1})$) solves locally a problem of the type 
\begin{equation}
\label{eq.thinobst_La}
  \left\{ \begin{array}{rcll}
  u & \ge & \varphi & \textrm{ on } B_1\cap \{x_{{n+1}} = 0\}\\
  L_a u&=&0 & \textrm{ in } B_1\setminus\left(\{x_{n+1} = 0\}\cap \{u = \varphi\}\right)\\
  L_a u & \le  &0 & \textrm{ in } B_1,
  \end{array}\right.
\end{equation}
that is, a thin obstacle problem with operator $L_a$, or a weighted thin obstacle problem (cf. \eqref{eq.thinobst_intro_3}) with $A_2$-Muckenhoupt weight. 

It is for this reason that many times one studies the weighted thin obstacle problem \eqref{eq.thinobst_La} with $a\in(-1,1)$ (see \cite{CS07, CSS08}). For the sake of simplicity and readability, in this introduction we will always assume $a = 0$, but most of the results mentioned generalize to any $a\in (-1, 1)$ accordingly, and therefore, they also apply to solutions to the fractional obstacle problem \eqref{eq.top_L_Om}.

Fractional obstacle problems such as \eqref{eq.top_L_Om}, as well as many of its variants (with more general non-local operators, with a drift term, in the parabolic case, etc.), have been a very prolific topic of research in the last years (see \cite{CF11, PP15, GP16, DGPT17, CRS17, BFR18b, FR18} and references therein). We refer the reader to the expository works \cite{PSU12, Sal12, DaSa18} for a deeper understanding of the fractional obstacle problem and its relation to the thin obstacle problem.

\subsubsection{The fractional Laplacian and L\'evy processes}

Integro-differential equations arise naturally in the study of stochastic processes with jumps, namely, L\'evy processes. The research in this area is attracting an increasing level of interest, both from an analytical and probabilistic point of view, among others, due to its applications to multiple areas: finance, population dynamics, physical and biological models, etc. (See \cite{DL76, Mer76, CT04, Ros16, Ros18} and references therein.) Infinitesimal generators\footnote{The infinitessimal generator $\mathcal{A}$ of a stochastic process $X = \{X_t : t\ge 0\}$, with $X_t\in \R^n$, is defined to act on suitable functions $f:\R^n\to \R$ as 
\[
\mathcal{A}(f(x)) := \lim_{t\downarrow 0} \frac{\mathbb{E}^x(f(X_t)) - f(x)}{t}.
\]} of L\'evy processes are integro-differential operators of the form
\begin{equation}
\label{eq.L}
\mathcal{L}u = b\cdot \nabla u +{\rm tr}\, (A\cdot D^2 u) +\int_{\R^n} \left\{u(x+y) -u(x) -y\cdot \nabla u(x)\chi_{B_1}(y)\right\}\nu(dy),
\end{equation}
for some L\'evy measure $\nu$ such that $\int \min\{1, |y|^2\}\nu(dy) < \infty$. The simplest (non-trivial) example of such infinitesimal generators is the fractional Laplacian introduced above, which arises as infinitesimal generator of a stable and radially symmetric L\'evy process. 

In particular, obstacle type problems involving general integro-differential operators of the form \eqref{eq.L} appear when studying the optimal stopping problem for a L\'evy process: consider a particle located at $X_t$ at time $t\ge 0$, moving according a L\'evy process inside a domain $\Omega$, and let $\varphi$ be a pay-off function defined in $\Omega$, and $\bar g$ an exterior condition defined in $\R^n\setminus \Omega$. At each time we can decide to stop the process and be paid $\varphi(X_t)$ or wait until the particle reaches a region where $\varphi$ has a higher value. Moreover, if the particle suddenly jumps outside of $\Omega$, we get paid $\bar g(X_t)$. The goal is to maximize the expected value of money we are being paid. We refer the interested reader to the aforementioned references as well as \cite{Pha97} and the appendix of \cite{BFR18} for the jump-diffusion optimal stopping problem, as well as \cite{LS09, Eva12, FR20} for the local (Brownian motion) case. 

\subsection{Regularity of the solution}

Once existence and uniqueness is established for solutions to \eqref{eq.thinobst_intro_2}, the next question that one wants to answer is: 
\[
\textrm{How regular is the solution $u$ to \eqref{eq.thinobst_intro_2}?}
\]

Regularity questions for solutions to the thin obstacle problem were first investigated by Lewy in \cite{Lew68}, where he showed, for the case $n = 1$, the continuity of the solution of the Signorini problem. He also gave the first proof related to the structure of the free boundary, by showing, also in $n = 1$, that if the obstacle $\varphi$ is concave, the coincidence set $\{u = \varphi\}$ consists of at most one connected interval. 

The continuity of the solution for any dimension follows from classical arguments. One first shows that the coincidence set $\{u = \varphi\}$ is closed, and then one uses the following fact for harmonic functions: if $\mathcal{C}\subset \Omega$ is closed, and $\Delta v = 0$ in $\Omega\setminus\mathcal{C}$ and $v$ is continuous on $\mathcal{C}$, then $v$ is continuous in $\Omega$. 

Rather simple arguments also yield that, in fact, the solution is Lipschitz. Indeed, if one considers the solution $u$ to the problem \eqref{eq.thinobst_intro_3}, and we define $h\in {\rm Lip}(B_1)$ as the solution to 
\[
  \left\{ \begin{array}{rcll}
  \Delta h&=&0 & \textrm{ in } B_1\setminus\{x_{n+1} = 0\}\\
    h&=&-\|u\|_{L^\infty(B_1)} & \textrm{ on } \de B_1 \\
  h & = &\varphi & \textrm{ on } B_1\cap \{x_{{n+1}} = 0\},
  \end{array}\right.
\]
then $u$ is a solution to the classical (thick) obstacle problem with $h$ (which is Lipschitz) as the obstacle. Then, we just notice that solutions to the thick obstacle problem with Lipschitz obstacles are Lipschitz, so $u$ is Lipschitz as well. This last step is not so immediate, we refer the reader to \cite[Theorem 1]{AC04} or \cite[Proposition~2.1]{Fer16} for two different ways to conclude this reasoning. These first regularity properties were investigated in the early 1970's (see \cite{Bei69, LS69, Kin71, BC72, GM75}).

In general, we do not expect solutions to \eqref{eq.thinobst_intro_3} to be better than Lipschitz in the full ball $B_1$. Indeed, across $\{x_{n+1} = 0\}$ on contact points, we have that normal derivatives can change sign, as seen by taking the even extension to \eqref{eq.thinobst_intro_2}. Nonetheless, we are interested in the regularity of the solution in either side of the obstacle. The fact that normal derivatives jump is \emph{artificial}, in the sense that it does not come from the equations, but from the geometry of the problem. We see that this is not observed in \eqref{eq.thinobst_intro_2}, where the solution could, a priori, be better than Lipschitz, and it also does not appear when studying the solution restricted to $\{x_{n+1} = 0\}$, as in the situations with the fractional obstacle problem \eqref{eq.top_f12_Om}. 

\subsubsection{$C^{1,\alpha}$ regularity}

The first step to upgrade the regularity of solutions to \eqref{eq.thinobst_intro_2} was taken by Frehse in \cite{Fre77}, where he proved that first derivatives of $u$ are continuous up to $\{x_{n+1} = 0\}$ on either side, thus showing that the solution is $C^1$ in $B_1^+$, up to the boundary.

Later, in 1978 Richardson proved that solutions are $C^{1,1/2}$ for $n = 1$ in \cite{Ric78}; whereas, in parallel, Caffarelli showed in \cite{Caf79} that solutions to the Signorini problem are $C^{1,\alpha}$ for some small $\alpha > 0$ up to the boundary on either side (alternatively, tangential derivatives are H\"older continuous).  In order to do that, Caffarelli showed the semiconvexity of the solution in the directions parallel to the thin obstacle. We state this result here for future reference, as well as the $C^{1,\alpha}$ regularity, and we refer the interested reader to \cite{Caf79, PSU12, Sal12, DaSa18} for the proofs of these results.
\begin{prop}[\cite{Caf79}]
\label{prop.semiconv}
Let $u$ be any weak solution to \eqref{eq.thinobst_intro_2}, and let $\varphi\in C^{1, 1}(B_1')$. Let $\be \in \mathbb{S}^{n}$ be parallel to the thin space, $\be \cdot \be_{n+1} = 0$. Then, $u$ is semiconvex in the $\be $ direction. That is, 
\[
\inf_{B_{1/2}} \de^2_{\be\be} u \ge -C (\|u\|_{L^2(B_1)}+[\nabla \varphi]_{C^{0,1}(B_1')}), 
\]
for some constant $C$ depending only on $n$. 
\end{prop}

As a (not immediate) consequence, Caffarelli deduced the $C^{1,\alpha}$ regularity of solutions.
\begin{thm}[\cite{Caf79}]
\label{thm.caf79}
Let $u$ be any weak solution to \eqref{eq.thinobst_intro_2}, and let $\varphi\in C^{1, 1}(B_1')$. Then, $u\in C^{1,\alpha}(\overline{B_{1/2}^+})$ and
\[
\|u\|_{C^{1,\alpha}(\overline{B_{1/2}^+})} \le C \left(\|u\|_{L^2(B_1^+)} + [\nabla \varphi]_{C^{0,1}(B_1')}\right),
\] 
for some constants $\alpha > 0$ and $C$ depending only on $n$. 
\end{thm}
\begin{rem}
In fact, Caffarelli in \cite{Caf79} pointed out how to deal with other smooth operators coming from variational inequalities with smooth coefficients.  Thus, in \eqref{eq.thinobst_intro_2} one could consider  other divergence form operators other than the Laplacian, with smooth and uniformly elliptic coefficients. 
\end{rem}
\begin{rem}
A posteriori, one can lower the regularity assumptions on the obstacle, the coefficients, and the lower dimensional manifold. We refer to \cite{RuSh17} for a study in this direction, with $C^{1,\alpha}$ obstacles, $C^{0,\alpha}$ coefficients (in divergence form), and with the thin obstacle supported on a $C^{1,\gamma}$ manifold. 
\end{rem}

The fact that the regularity cannot be better than $C^{1,1/2}$ is due to the simple counter-example,
\begin{equation}
\label{eq.halfspacesol}
u(x) = {\rm Re}\left((x_1 + i |x_{n+1}|)^{3/2}\right)
\end{equation}
which in $(x_1, x_{n+1})$-polar coordinates can be written as 
\[
\tilde u(r, \theta) = r^{3/2} \cos \left(\textstyle{\frac{3}{2}}\theta\right). 
\]
The function $u$ is a solution to the Signorini problem: it is harmonic for $|x_{n+1}| > 0$, the normal derivative $\de_{x_{n+1}}$ vanishes at $\theta = 0$, and has the right sign at $\theta = \pi$. 

It was not until many years later that, in \cite{AC04}, Athanasopoulos and Caffarelli showed the optimal $C^{1, 1/2}$ regularity of the solution in all dimensions. That is, in the previous theorem $\alpha = \frac12$, and by the example above, this is optimal. We leave the discussion of the optimal regularity for the next section, where we deal with the classification of free boundary points. 

Historically, the classification of the free boundary was performed \emph{after} having established the optimal regularity. In the next section we show that this is not needed, and in fact one can first study the free boundary, and from that deduce the optimal regularity of the solution. 

\section{Classification of free boundary points and optimal regularity}
\label{sec.6}
The thin obstacle problem, \eqref{eq.thinobst_intro_2} or \eqref{eq.thinobst_intro_3}, is a {\em free boundary problem}, i.e., the unknowns of the problem are the solution itself, and the contact set
\[
\Lambda(u) := \big\{x'\in \R^n: u(x', 0) = \varphi(x')\big\}\times\{0\}\subset \R^{n+1},
\]
 whose topological boundary in the relative topology of $\R^n$, which we denote $\Gamma(u) = \de_{\R^n}\Lambda(u) = \de \{x'\in \R^n: u(x', 0) = \varphi(x')\}\times\{0\}$, is known as the \emph{free boundary}.

After studying the basic regularity of the solution, the next natural step in understanding the thin obstacle problem is the study of the structure and regularity of the free boundary. This is also related to the optimal regularity question presented above, since one expects that the \emph{worst} points in terms of regularity lie on the free boundary.

Let us suppose, for simplicity, that we have a zero obstacle problem, $\varphi \equiv 0$. Notice that, if the obstacle $\varphi$ is analytic, we can always reduce to this case by subtracting an even harmonic extension of $\varphi$ to the solution\footnote{If the obstacle $\varphi$ is analytic, then $\varphi$ has a harmonic extension to $B_1^+$, and its even extension in the whole $B_1$ is harmonic as well. Thus, the function $u-\varphi$ solves a thin obstacle problem with zero obstacle. This is no longer true if $\varphi$ is not analytic (not even when $\varphi\in C^\infty$), and in such situation one needs to adapt the arguments. However, the main ideas are the same.}. This is not possible under lower regularity properties (in particular, this does not include the case where $\varphi\in C^\infty$, see Section~\ref{sec.cinf}). We will, moreover, assume that we deal with an even solution (for example, by imposing an even boundary datum, see Remark~\ref{rem.geven}). 

Our problem is
\begin{equation}
\label{eq.TOP0}
  \left\{ \begin{array}{rcll}
  u & \ge & 0 & \textrm{ on }B_1 \cap \{x_{n+1} = 0\}\\
  \Delta u&=&0 & \textrm{ in } B_1\setminus\left(\{x_{n+1} = 0\}\cap \{u = 0\}\right)\\
  \Delta u & \le & 0 & \textrm{ in } B_1\\
  &&&\hspace{-2cm} u \text{ is even in the $x_{n+1}$ direction},
  \end{array}\right.
\end{equation}
and the contact set is
\[
\Lambda(u) = \{(x', 0)\in \R^{n+1} : u(x', 0) = 0\}.  
\]

In order to study a free boundary point, $x_\circ\in \Gamma(u)$, one considers \emph{blow-ups} of the solution $u$ around $x_\circ$. That is, one looks at rescalings of the form
\begin{equation}
\label{eq.ur}
u_{r,x_\circ}(x)  = \frac{u(x_\circ + rx)}{\left(\ave_{\de B_r(x_\circ)} u^2\right)^{\frac12}}.
\end{equation}
The limit of such rescalings, as $r\downarrow 0$, gives information about the behaviour of the solution around the free boundary point $x_\circ$. Thus, classifying possible blow-up profiles as $r\downarrow 0$ around free boundary points will help us to better understand the structure of the free boundary. Notice that, by construction, the blow-up sequence \eqref{eq.ur} is trivially bounded in $L^2(\de B_1)$. To prove (stronger) convergence results, we need the sequence to be bounded in more restrictive spaces (say, in $W^{1, 2}$), by taking advantage of the fact that $u$ solves problem \eqref{eq.TOP0}. 

In order to do that, a very powerful tool is \emph{Almgren's frequency function}. If we consider a solution $u$ to the Signorini problem \eqref{eq.TOP0} and take the odd extension (with respect to $x_{n+1}$), we end up with a two-valued map that is harmonic on the thin space (and has two branches). Almgren studied  in \cite{Alm00} precisely the monotonicity of the frequency function for multi-valued harmonic functions (in fact, Dirichlet energy minimizers), and thus, it is not surprising that such tool is also available in this setting. 

Let us define, for a free boundary point $x_\circ \in \Gamma(u)$, 
\[
N(r, u, x_\circ) := \frac{r\int_{B_r(x_\circ)} |\nabla u|^2}{\int_{\de B_r(x_\circ)} u^2}.
\]
We will often denote $N(r, u)$ whenever we take $x_\circ = 0$. Notice that $N(\rho, u_r) = N(r\rho, u)$, where $u_r := u_{r, 0}$ (see \eqref{eq.ur}). Then, we have the following.
\begin{lem}
\label{lem.Almgren}
Let $u$ be a solution to \eqref{eq.TOP0}, and let us assume $0\in \Gamma(u)$. Then, Almgren's frequency function 
\[
r\mapsto N(r, u) = \frac{r\int_{B_r} |\nabla u|^2}{\int_{\de B_r}u^2}
\]
is nondecreasing. Moreover, $N(r, u)$ is constant if and only if $u$ is homogeneous.
\end{lem}
\begin{proof}
We very briefly sketch the proof (see \cite{PSU12, Sal12, DaSa18} for more details). By scaling ($N(\rho, u_r) = N(r\rho, u)$) it is enough to show that $N'(1, u)\ge 0$. Let us denote 
\[
 D(r, u) = \frac{1}{r^{n+1}}\int_{B_r}|\nabla u|^2 = r^2\int_{B_1}|\nabla u(r\cdot )|^2,\quad H(r, u) = \frac{1}{r^n}\int_{\de B_r} u^2 = \int_{\de B_1} u(r\cdot)^2,
\]
so that $N(r, u ) = \frac{D(r, u)}{H(r, u)}$ and 
$
N'(1, u) = \frac{D(1, u)}{H(1, u)}\left(\frac{D'(1, u)}{D(1, u)} - \frac{H'(1, u)}{H(1, u)}\right)
$.
Now notice that 
\begin{align*}
D'(1, u)=  2\int_{B_1} \nabla u \cdot \nabla (x\cdot\nabla u) \, dx = 2\int_{\de B_1} u^2_{\nu} - 2\int_{B_1}\Delta u (x\cdot \nabla u)\, dx,
\end{align*}
%\begin{align*}
%D'(1, u) & = 2\int_{B_1} |\nabla u|^2 + 2\int_{B_1} \nabla u \cdot D^2 u\cdot x \, dx = 2\int_{B_1} \nabla u \cdot \nabla (x\cdot\nabla u) \, dx\\
%& = 2\int_{\de B_1} u^2_{\nu} - 2\int_{B_1}\Delta u (x\cdot \nabla u)\, dx,
%\end{align*}
where $u_\nu$ denotes the outward normal derivative to $B_1$. Since $u$ is a solution to the Signorini problem, either $\Delta u = 0$ or $u = 0$ and $\Delta u > 0$ (in which case, $x\cdot\nabla u = 0$ by $C^1$ regularity of the solution). Thus, the second term above vanishes. On the other hand, we have that
$
H'(1, u) = 2\int_{\de B_1} u u_\nu$ and $D(1, u) = \int_{B_1}|\nabla u|^2 = \int_{\de B_1} u u_\nu,
$
where in the last equality we have used again that $u$ solves the Signorini problem, and therefore $u\Delta u \equiv 0$. 
Thus, 
\[
N'(1, u) = 2\frac{D(1, u)}{H(1, u)}\left(\frac{\int_{\de B_1} u_\nu^2}{\int_{\de B_1} u u_\nu} - \frac{\int_{\de B_1} u u_\nu}{\int_{\de B_1}u^2}\right)\ge 0, 
\]
by Cauchy-Schwarz inequality. Equality holds if and only if $u$ is proportional to $u_\nu$ on $\de B_r$ for every $r$ (that is, $u$ is homogeneous). 
\end{proof}

In particular, as a consequence of Lemma~\ref{lem.Almgren} we have that $\lambda := \lim_{r\downarrow 0}N(r, u) = N(0^+, u)$ is well-defined. This value is known as the \emph{frequency} at a free boundary point. 

From Lemma~\ref{lem.Almgren} we also have the following.

\begin{lem}
\label{lem.Almgren2}
Let $u$ be a solution to \eqref{eq.TOP0}, and let us assume $0\in \Gamma(u)$. Let $\lambda := N(0^+, u)$, and let
\[
\psi(r) := \ave_{\de B_r} u^2.
\]
Then, the function $r\mapsto r^{-2\lambda}\psi(r)$ is nondecreasing. Moreover, for every $\eps > 0$ there exists some $r_\circ = r_\circ(\eps)$ such that if $r < \rho r \le r_\circ(\eps)\le 1$,
\[
\psi(\rho r) \le \rho^{2(\lambda +\eps)} \psi(r).
\]
\end{lem}
\begin{proof}
Differentiating
\[
\frac{d}{dr} \left(r^{-2\lambda}\psi(r) \right) = 2r^{-2\lambda-n-1}\left\{r\int_{B_r}|\nabla u|^2 - \lambda \int_{\de B_r} u^2\right\}\ge 0,
\]
where we are also using the monotonicity of $N(r, u)$ from Lemma~\ref{lem.Almgren}. 

On the other hand, choose $r_\circ(\eps)$ such that $N(r_\circ, u)\le \lambda +\eps$.  Then, just noticing that 
\begin{equation}
\label{eq.compwith}
N(r, u) = \frac{r}{2}\frac{d}{dr} \log \psi(r)\le \lambda+\eps
\end{equation}
for $r < \rho r \le r_\circ$ and integrating in $(r, \rho r)$ we get the desired result.
\end{proof}

As a consequence of Almgren's monotonicity formula we get the existence of a (homogeneous) blow-up limit around free boundary points, $u_0$. Notice that we are not claiming the uniqueness of such blow-up, but its degree of homogeneity is independent of the sequence. 
\begin{cor}
\label{cor.strongconv}
Let $u$ be a solution to \eqref{eq.TOP0}, and let us assume $0\in \Gamma(u)$. Let us denote
\[
u_r(x) = \frac{u(rx)}{\left(\ave_{\de B_r} u^2\right)^{1/2}}.
\]
Then, for any sequence $r_k\downarrow 0$ there exists a subsequence $r_{k_j}\downarrow 0$ such that
\begin{align}
u_{r_{k_j}} & \to u_0\quad\textrm{strongly in } L^2_{\rm loc} (\R^{n+1}),\\
\nabla u_{r_{k_j}} & \rightharpoonup \nabla u_0\quad\textrm{weakly in } L^2_{\rm loc}(\R^{n+1}),\\
\label{eq.strconv} u_{r_{k_j}} & \to u_0\quad\textrm{strongly in } C^1_{\rm loc}(\overline{\R^{n+1}_+}),
\end{align}
for some $N(0^+, u)$-homogeneous global solution $u_0$ to the thin obstacle problem with zero obstacle, \eqref{eq.TOP0}, and $\|u_0\|_{L^2(\partial B_1)} = c_n$, for some dimensional constant $c_n  >0$. Moreover $u_0\in C^{1,\alpha}_{\rm loc}(\{x_{n+1} \ge 0\})$ and $N(0^+, u) > 1$. 
\end{cor}
\begin{proof}
The proof of the strong convergence in $L^2$ and weak convergence in $W^{1, 2}$ is a consequence of Lemma~\ref{lem.Almgren}, which shows that the sequence $u_{r_k}$ is uniformly bounded in $W^{1,2}(B_1)$. Indeed, take any ball centered at the origin, $B_R\subset \R^n$. Then, using the notation from Lemma~\ref{lem.Almgren2},
\[
\int_{B_R}|\nabla u_r|^2 = \frac{r^{1-n}}{\psi(r)}\int_{B_{rR}} |\nabla u|^2\le \frac{R^{n-1}\psi(Rr)}{\psi(r)} N(1, u)\le C(R) N(1, u),
\]
where in the last step we are using that $r$ is small enough together with the second part of Lemma~\ref{lem.Almgren2} with $\eps = 1$. Also notice that $\|u_r\|_{L^2(\partial  B_1)} = c_n$, so $u_r$ is bounded in $W^{1,2}$ for every compact set (again, by Lemma~\ref{lem.Almgren2}).

The homogeneity of $u_0$ comes from the fact that 
\[
N(\rho, u_0) = \lim_{r\downarrow 0} N(\rho, u_r) = \lim_{r\downarrow 0} N(r\rho, u) = N(0^+, u),
\] 
and Lemma~\ref{lem.Almgren}. The strong convergence in $C^1$ follows from the $C^{1,\alpha}$ regularity estimates for the solution, Theorem~\ref{thm.caf79}, and it implies that $u_0 \in C^{1,\alpha}_{\rm loc}(\{x_{n+1} \ge 0\})$ solves the thin obstacle problem \eqref{eq.TOP0} in $\R^{n+1}$, and $N(0^+, u_0) > 1$.
\end{proof}

%Hence, we obtain the following immediate corollary, describing the structure of blow-ups at free boundary points.

%\begin{cor}
%\label{cor.blowupthm}
%Let $u$ be a solution to \eqref{eq.TOP0}, and let us assume $0\in \Gamma(u)$. Let $u_0$ denote any blow-up at $0$. Then, $u_0$ satisfies
%\[
%\left\{
%\begin{array}{l}
%u_0\in C^{1,\alpha}_{\rm loc}(\{x_{n+1} \ge 0\})\\
%u_0 \textrm{ solves the thin obstacle problem \eqref{eq.TOP0} in }\R^{n+1}\\
%u_0 \textrm{ is $\lambda$-homogeneous, with }\lambda>1. 
%\end{array}
%\right.
%\]
%\end{cor}
%\begin{proof}
%The fact that $u_0 \in C^{1,\alpha}_{\rm loc}(\{x_{n+1} \ge 0\})$ solves the thin obstacle problem \eqref{eq.TOP0} in $\R^{n+1}$ comes directly from the strong convergence \eqref{eq.strconv}. Also, from Corollary~\ref{cor.strongconv}, $u_0$ is a $\lambda:= N(0^+, u)$ homogeneous function, and from the $C^1$ convergence of the blow-ups, it is clear that $\lambda > 1$.
%\end{proof}

And we can now classify blow-ups at free boundary points.

\begin{thm}[Classification of blow-ups]
\label{thm.blowupthm}
Let $u$ be a solution to \eqref{eq.TOP0}, and let us assume $0\in \Gamma(u)$. Let $u_0$ denote any blow-up at $0$ with homogeneity $\lambda := N(0^+, u)$. Then $\lambda\in \left\{\frac32\right\}\cup [2, \infty)$. Moreover, if $ \lambda = \frac32$, then $u_0$ is (after a rotation) of the form \eqref{eq.halfspacesol}. 
\end{thm}
\begin{proof}
We need to determine the possible values $\lambda = N(0^+,u)$ in Corollary~\ref{cor.strongconv} whenever $\lambda < 2$. 

Thus, from now on, let us assume that $1<\lambda < 2$.  We separate the rest of the proof into two steps.
\\[0.2cm]
{\it Step 1: Convexity of $u_0$.} Let us start by showing that $u_0$ is convex in the directions parallel to the thin space, and thus, in particular, the restriction $u_0 |_{\{x_{n+1} = 0\}}$ is convex. We do so by means of the semiconvexity estimates from Proposition~\ref{prop.semiconv} applied to $u_0$. Indeed, by rescaling Proposition~\ref{prop.semiconv} to a ball of radius $R \ge 1$ we get 
\[
R^2 \inf_{B_{R/2}} \partial_{\be\be} u_0 \ge - C R^{-\frac{n}{2}} \|u_0\|_{L^2(B_R)} = -CR^\lambda\|u_0\|_{L^2(B_1)},
\]
for some dimensional constant $C$, and for $\be \cdot \be_{n+1} = 0$, where in the last equality we are using the $\lambda$-homogeneity of $u_0$. That is, by letting $R \to \infty$,
\[
\inf_{B_{R/2}} \partial_{\be\be} u_0 \ge- C R^{\lambda - 2}\|u_0\|_{L^2(B_1)}\to 0,\quad\text{as}\quad R\to \infty.
\]
%
%Indeed, using the notation from Lemma~\ref{lem.Almgren2} together with Proposition~\ref{prop.semiconv}, we have that 
%\[
%\inf_{B_{1/(2r)}} \de^2_{\be\be} u_r \ge -\frac{Cr^2}{\psi(r)^{1/2}}\|u\|_{L^2(B_1)},
%\]
%for some dimensional constant $C$, and for $\be \cdot \be_{n+1} = 0$. On the other hand, take $\eps_\circ$ such that $\lambda + \eps_\circ < 2$, and consider $r_\circ = r_\circ(\eps_\circ)$ from Lemma~\ref{lem.Almgren2}. Then, by Lemma~\ref{lem.Almgren2}, 
%\[
%r^{2(\lambda+\eps_\circ)} \frac{\psi(r_\circ) }{r_\circ^{2(\lambda+\eps_\circ)} }\le \psi(r)
%\]
%for $ r \le r_\circ$. That is, 
%\[
%\inf_{B_{1/(2r)}} \de^2_{\be\be} u_r \ge -Cr_\circ^{\lambda+\eps_\circ}\frac{r^{2-\lambda -\eps_\circ}}{\psi(r_\circ)^{1/2}}\|u\|_{L^2(B_1)}\to 0\quad\textrm{as}\quad r\downarrow 0. 
%\]
Hence, $u_0$ is convex in the directions tangential to the thin space. 
\\[0.2cm]
{\it Step 2: Degree of homogeneity of $u_0$}. Let us now consider $\Lambda(u_0)\subset \{x_{n+1} = 0\}$ the contact set for $u_0$, which is a convex cone, from the convexity and homogeneity of $u_0$. 

If $\Lambda(u_0)$ has empty interior (restricted to the thin space), then $\de_{x_{n+1}} u_0$ is a harmonic function in $\{x_{n+1} > 0\}$, identically zero on the thin space, and $(\lambda- 1)$-homogeneous. In particular, from the sublinear growth at infinity,  $\de_{x_{n+1}} u_0\equiv 0$ everywhere, and thus $u_0 \equiv 0$, a contradiction. Hence, $\Lambda(u_0)$ has non-empty interior on the thin space. 

Let us denote $\be \in \mathbb{S}^{n-1}$ a direction contained in the interior of $\Lambda(u_0)$ (in particular, $\be\cdot\be_{n+1}  = 0$). Let us define $w_1 := \de_{-\be} u_0$, and $w_2 := -\de_{x_{n+1}} u_0$ for $x_{n+1} \ge 0$, and $w_2 := \de_{x_{n+1}} u_0$ for $x_{n+1}< 0$. Notice that since $u_0$ is even and $\partial_{x_{n+1}} u_0 = 0$ whenever $u_0 > 0$ on the thin space, $w_2$ is continuous across $\{x_{n+1} = 0\}$. Moreover, $w_1$ and $w_2$ are $(\lambda-1)$ homogeneous functions, harmonic in $\{x_{n+1} \neq 0\}$

%Let us define $w_1 := \de_{-\be} u_0$, and $w_2 := -|\de_{x_{n+1}} u_0|$, which are $(\lambda-1)$ homogeneous functions, harmonic in $\{x_{n+1} \neq 0\}$. 

Notice that $w_1 = 0$ in $\Lambda(u_0)$. In particular, for any $x_\circ\in \{x_{n+1} = 0\}$, $x_\circ + t\be\in \Lambda(u_0)$ for $t\in \R$ large enough (since $\Lambda(u_0)$ is a cone with non-empty interior and $\be$ is a direction contained in it). Thus, from the convexity of $u_0$, $w_1$ has to be monotone along $x_\circ+t\be$, and thus $w_1 \ge 0$ on the thin space. Since $w_1$ is $(\lambda-1)$-homogeneous (i.e., it has sublinear growth), and is non-negative on the thin space, there is a unique $(\lambda-1)$-homogeneous harmonic extension that coincides with $w_1$ (by the Poisson kernel), and it is non-negative as well. Hence, $w_1\ge 0$ in $\R^{n+1}$. 

In addition, $w_2\ge 0$ on the thin space as well (since $u_0$ solves the thin obstacle problem), and it has sublinear growth at infinity. That is, its harmonic extension is itself, and thus $w_2 \ge 0$ in $\R^{n+1}$. Moreover, notice that $w_2 = 0$ in $\{x_{n+1} = 0\}\setminus \Lambda(u_0)$ (in particular, $w_1 w_2 \equiv 0$ on $\{x_{n+1} = 0\}$). 

On the one hand, we have that the restriction of $w_1$ to the unit sphere must be the first eigenfunction of the Dirichlet problem for the spherical Laplacian  with zero data on $\de B_1 \cap \Lambda(u_0)$ (since it is non-negative), and it has homogeneity $\lambda-1$. On the other hand, the restriction of $w_2$ to the unit sphere must be the first eigenfunction with zero data on $\de B_1 \cap (\{x_{n+1}= 0\}\setminus \Lambda(u_0))$, and it has the same homogeneity $\lambda-1$. Since $\Lambda(u_0)$ is a (convex) cone, it is contained in a half-space (of $\{x_{n+1} = 0\}$), and therefore, $\{x_{n+1} = 0\}\setminus \Lambda(u_0)$ contains a half-space. Since the corresponding homogeneities are the same (i.e., $\lambda-1$), by monotonicity of eigenvalues with respect to the domain we must have that, after a rotation, $\Lambda(u_0)$ and its complement are equal, and hence, they are half-spaces themselves. The homogeneity for the half-space in this situation is $\frac12$, so $\lambda = \frac32$, and the corresponding eigenfunction is
\[
u_0(x) = {\rm Re}\left((x_1 + i |x_{n+1}|)^{3/2}\right),
\]
as we wanted to see. 
\end{proof}

As a consequence of the previous result, we have a dichotomy for free boundary points. 
\begin{prop}[Classification of free boundary points and optimal regularity]
Let $u$ be a solution to \eqref{eq.TOP0}. Then, the free boundary can be divided into two sets, 
\[
\Gamma(u) = {\rm Reg}(u) \cup {\rm Deg}(u).
\]
The set of \emph{regular points}, 
\[
{\rm Reg}(u) := \left\{x_\circ \in \Gamma(u) : N(0^+, u, x_\circ) = {\textstyle \frac32}\right\},
\]
and the set of \emph{degenerate points},
\[
{\rm Deg}(u) := \left\{x_\circ \in \Gamma(u) : N(0^+, u, x_\circ) \ge 2\right\}.
\]
Moreover, $u\in C^{1, 1/2}(B_{1}^+)$ with 
\begin{equation}
\label{eq.intro_gives2}
\|u\|_{C^{1,1/2}(\overline{B_{1/2}^+})}  \le C \|u\|_{L^\infty(B_1)}
\end{equation}
for some $C$ depending only on $n$, and the set of regular points is open (in the relative topology of the free boundary). 
\end{prop}
\begin{proof}
The classification result is an immediate consequence of Corollary~\ref{cor.strongconv} and Theorem~\ref{thm.blowupthm}.

For the optimal regularity, we observe that by Corollary~\ref{cor.strongconv}, since the sequence $u_r$ is uniformly bounded in $r$, for $x_\circ\in \Gamma(u)$,
\begin{equation}
\label{eq.intro_growthu}
\|u\|_{L^\infty(B_r)(x_\circ)} \le C \left(\ave_{\de B_r(x_\circ)} u^2\right)^{\frac12}\le C \|u\|_{L^\infty(B_1)}r^{\frac32},
\end{equation}
where in the last inequality we are using Lemma~\ref{lem.Almgren2}, together with the fact that, by Theorem~\ref{thm.blowupthm}, $N(0^+, u, x_\circ)\ge\frac32$. This establishes a uniform growth of the solution around free boundary points. Combined with interior estimates for harmonic functions, this yields that $u$ is $C^{1, 1/2}$ on the thin space, and thus $u\in C^{1, 1/2}(B_1^+)$ with estimates in $B_{1/2}^+$. 

Indeed, take $y \in \{x_{n+1} = 0\}\cap \{u > 0\}\cap B_{3/4}$, and let $r = {\rm dist}(y, \Gamma(u)) = {\rm dist}(y, y_\circ)$ for some $y_\circ\in \Gamma(u)$. Then $u$ is harmonic in $B_r(y)$, and by harmonic estimates together with \eqref{eq.intro_growthu} applied at $B_{2r}(y_\circ)$
\begin{equation}
\label{eq.intro_growthnablau}
\|\nabla_{x'} u\|_{L^\infty(B_{r/2}(y))} \le Cr^{-1}\|u\|_{L^\infty(B_r(y))}\le Cr^{-1}\|u\|_{L^\infty(B_{2r}(y_\circ))}\le C \|u\|_{L^\infty(B_1)}r^{\frac12}.
\end{equation}
%In particular, applying the previous estimate at every point $y\in B_r(x_\circ)$ and noticing $\nabla_{x'} u \equiv 0$ on the contact set $\{x_{n+1} = 0\}\cap \{u = 0\}$,
%\begin{equation}
%\label{eq.intro_growthnablau}
%\|\nabla_{x'} u\|_{L^\infty(B_{r}(x_\circ))}  \le C\|u\|_{L^\infty(B_1)} r^\frac12
%\end{equation}
%for $x_\circ \in \Gamma(u)$. 
Next for $y_1, y_2\in \{x_{n+1} = 0\}\cap B_{3/4}$ we want to obtain the bound 
\begin{equation}
\label{eq.intro_gives}
|\nabla_{x'} u(y_1) - \nabla_{x'} u(y_2)|\le C\|u\|_{L^\infty(B_1)}|y_1 - y_2|^{\frac12}
\end{equation}
to get $C^{1,1/2}$ regularity of $u$ on the thin space. Notice that, since $\nabla_{x'} u = 0$ on $\{x_{n+1} = 0\}\cap \{u = 0\}$, we can assume that $y_1,y_2\in \{x_{n+1} = 0\} \cap \{ u > 0\}\cap B_{3/4}$. 

Let us suppose  $r = {\rm dist}(y_1, \Gamma(u)) \ge {\rm dist}(y_2, \Gamma(u))$. Then, if ${\rm dist}(y_1, y_2) \le \frac{r}{2}$, and since $u$ is harmonic in $B_r(y_1)$, by harmonic estimates we have 
\[
\frac{|\nabla_{x'}u(y_1)- \nabla_{x'} u(y_2)|}{|y_1-y_2|^{1/2}}\le [\nabla_{x'} u ]_{C^{1/2}(B_{r/2}(y_1))}\le Cr^{-1/2}\|\nabla_{x'}u\|_{L^\infty(B_r(y_1))}\le C\|u\|_{L^\infty(B_1)}
\]
where in the last step we have used \eqref{eq.intro_growthnablau}. On the other hand, if ${\rm dist}(y_1, y_2) \ge \frac{r}{2}$, from \eqref{eq.intro_growthnablau} and ${\rm dist}(y_2, \Gamma(u)) \le r$, 
\begin{align*}
|\nabla_{x'}u(y_1)- \nabla_{x'} u(y_2)|& \le |\nabla_{x'}u(y_1)|+|\nabla_{x'} u(y_2)|\\
& \le C\|u\|_{L^\infty(B_1)}r^{1/2}\le C\|u\|_{L^\infty(B_1)} |y_1-y_2|^{1/2}. 
\end{align*}
In all, \eqref{eq.intro_gives} always holds, and $u$ is $C^{1, 1/2}$ on $\{x_{n+1} = 0\}$. By standard harmonic estimates, its harmonic extension to $B_1^+$ is also $C^{1,1/2}$ with estimates up to the boundary $\{x_{n+1} = 0\}$, which gives \eqref{eq.intro_gives2}.

Finally, we note that the functions $\Gamma(u) \ni x\mapsto N(r, u, x)$ continuous for every $r > 0$ fixed, and are monotone nondecreasing in $r>0$ (for $x$ fixed). Thus, $N(0^+, u, x) = \inf_{r > 0} N(r, u, x)$ is the infimum of a family of continuous functions, and therefore, it is upper semi-continuous. In particular, if ${\rm Deg}(u)\ni x_k\to x_\circ$, then $N(0^+, u, x_\circ) \ge \limsup_{k\to \infty} N(0^+, u, x_k) \ge 2$, and thus $x_\circ\in {\rm Deg}(u)$. The set of degenerate points closed, and the set of regular points is open (in the relative topology of the free boundary).
\end{proof}

\section{Regular points}
\label{sec.7}
We have shown that  the free boundary can be divided into two different sets: regular points, and degenerate points, according to the value of the frequency. 

As we will show next, the set of regular points is so called because the free boundary is smooth around them, \cite{ACS08}.  

Let 0 be a regular free boundary point, and consider the rescalings 
\begin{equation}
\label{eq.blowupseq}
u_r(x) = \frac{u(rx)}{\left(\ave_{\de B_r} u^2\right)^{\frac12}}.
\end{equation}
Since 0 is a regular point, by Theorem~\ref{thm.blowupthm}, there exists some sequence $r_j\downarrow 0$ such that, up to a rotation, 
\begin{equation}
\label{eq.hslimit}
u_{r_j}\to u_0:= {\rm Re} \left( (x_1 +i|x_{n+1}| )^{3/2} \right)\quad\textrm{strongly in $C^1(B_{1/2}^+)$}.
\end{equation}
Notice that, on the thin space, $u_0$ is a half-space solution of the form $u_0(x', 0) = c(x_1)_+^{3/2}$. In particular, the free boundary is a hyperplane (in $\{x_{n+1} = 0\}$) and thus smooth. We want to show that the smoothness of the free boundary in the limit is inherited by the approximating sequence, $u_{r_j}$, for $j$ large enough. 

Let us start by showing a sketch of the proof of the fact that the free boundary is Lipschitz. In the following proposition, $\mathcal{C}(\be_1, \theta)$ denotes a cone in the tangential directions with axis $\be_1$ and opening $\theta > 0$, that is
\[
\mathcal{C}(\be_1, \theta) := \left\{\tau \in \R^{n+1}: \tau_{n+1} = 0,\,\tau\cdot \be_1 \ge \cos(\theta)\|\tau\|\right\}.
\]
\begin{prop}
\label{prop.lipfb}
Let $u$ be a solution to \eqref{eq.TOP0}, and suppose that the origin is a regular free boundary point,  i.e. $0\in {\rm Reg(u)}$. Suppose also that \eqref{eq.hslimit} holds. 

Then, for any fixed $\theta_\circ > 0$, there exists some $\rho > 0$ such that 
\begin{equation}
\label{eq.lip0}
\de_{\tau} u  \ge 0\quad \textrm{ in $B_\rho$, for all $\tau\in \mathcal{C}(\be_1, \theta_\circ)$}.
\end{equation}
In particular, the free boundary is Lipschitz around regular points. That is, for some neighbourhood of the origin, $\Gamma(u)$ is the graph of a Lipschitz function $x_1 = f(x_2,\dots,x_n)$ in $\{x_{n+1} = 0\}$. 
\end{prop}
\begin{proof}[Sketch of the proof]
We use that $\de_\tau u_{r_j}$ is converging to $\de_\tau u_0$ uniformly in $B_{1/2}$. Notice that, by assumption, $\de_\tau u_0\ge 0$, and in fact, $\de_\tau u_0 \ge c(\theta_\circ, \sigma) > 0$ in $\{|x_{n+1}| > \sigma\}$.

Thus, from the uniform convergence, for any $\sigma  > 0$ there exists some $r_\circ = r_\circ(\theta_\circ, \sigma)$ such that, if $r_j \le r_\circ$, 
\begin{equation}
\label{eq.lip1}
\begin{array}{ll}
\de_\tau u_{r_j}  \ge 0 & \textrm{ in } B_{3/4}\setminus \{|x_{n+1}|\ge \sigma\}\\
\de_\tau u_{r_j}  \ge c(\theta_\circ) > 0& \textrm{ in } B_{3/4}\setminus \left\{|x_{n+1}|\ge \textstyle{\frac12}\right\}.
\end{array}
\end{equation}

Moreover, from the optimal $C^{1,\frac12}$ regularity of solutions,
\begin{equation}
\label{eq.lip2}
\de_\tau u_{r_j} \ge - c \sigma^{\frac12}\quad \textrm{ in } B_{3/4}\cap \{|x_{n+1} \le \sigma\}.
\end{equation}

Combining \eqref{eq.lip1}-\eqref{eq.lip2} with the fact that $\Delta (\de_\tau u_{r_j}) = 0$ in $B_1\setminus \Lambda(u_{r_j})$, and $\de_\tau u_{r_j} = 0$ on $\Lambda(u_{r_j})$, by comparison principle arguments we deduce that there exist some $\sigma_\circ = \sigma_\circ(\theta_\circ)$ such that if $\sigma < \sigma_\circ$, $\de_\tau u_{r_j} \ge 0$ in $B_{1/2}$ (see \cite[Lemma 5]{ACS08} for a proper justification of this step). In particular, there exists some $\rho$ (depending only on $\theta_\circ$, but also depending on the regular point) such that $\de_\tau u_{\rho} \ge 0$ in $B_1$. Thus, \eqref{eq.lip0} holds. 

We finish by showing that \eqref{eq.lip0} implies that the free boundary is Lipschitz. We do so by considering the two (half) cones
\[
\Sigma_\pm := \pm \mathcal{C}(\be_1, \theta_\circ)\cap B_{\rho/2}.
\]

Notice that, since $0\in \Gamma(u)$, $u(0) = 0$, and from $u \ge 0$ on $\{x_{n+1} = 0\}$ together with \eqref{eq.lip0} we must have $u \equiv 0$ on $\Sigma_-$, so $\Sigma_-\subset \{u = 0\}$.

On the other hand, suppose that $y_\circ\in \Sigma_+$ is such that $u(y_\circ) = 0$. Again, by \eqref{eq.lip0} and the non-negativity of $u$ on the thin space, we have $u \equiv 0$ on $y_\circ - \mathcal{C}(\be_1, \theta_\circ)$. But notice that, since $y_\circ\in \Sigma_+$, $0\in y_\circ - \mathcal{C}(\be_1, \theta_\circ)$, that is, $0$ is not a free boundary point. A contradiction. Therefore, we have that $u(y_\circ) > 0$, so $\Sigma_+\subset \{u > 0\}$. 

Thus, the free boundary at 0 has a cone touching from above and below, and therefore, it is Lipschitz at the origin. We can do the same at the other points around it, so that the free boundary is Lipschitz. 
\end{proof}

In fact, the previous proof not only shows that the free boundary is Lipschitz, but letting $\theta_\circ\downarrow 0$ we are showing that it is basically $C^1$ at 0. In order to upgrade the regularity of the free boundary around regular points we use the following boundary Harnack principle. 

\begin{thm}[Boundary Harnack Principle, \cite{ACS08, DS19}]
\label{thm.bhp}
Let $\Omega\subset \{x_{n+1} = 0\}\cap B_1$ be any open set on the thin space, and let $v_1, v_2\in C(B_1)$ satisfying $\Delta v_1 = \Delta v_2 = 0$ in $B_1\setminus \Omega$. Assume, moreover, that $v_1$ and $v_2$ vanish continuously on $\Omega$, and $v_1, v_2 > 0$ in $B_1\setminus \Omega$. Then, there exists some $\alpha > 0 $ such that $\frac{v_1}{v_2}$ is $\alpha$-H\"older continuous in $B_{1/2}\setminus \Omega$ up to $\Omega$. 
\end{thm}
%that $v_1(\textstyle{\frac12}\be_{n+1}) = v_2(\textstyle{\frac12}\be_{n+1}) = 1$

As a consequence, we can show that the Lipschitz part of the free boundary is, in fact, $C^{1,\alpha}$. 

\begin{thm}[$C^{1,\alpha}$ regularity of the free boundary around regular points]
\label{thm.C1a}
Let $u$ be a solution to \eqref{eq.TOP0}. Then, the set of regular points, ${\rm Reg}(u)$, is locally a $C^{1,\alpha}$ $(n-1)$-dimensional manifold.  
\end{thm}
\begin{proof}
We apply Theorem~\ref{thm.bhp} to the right functions. By Proposition~\ref{prop.lipfb} we already know that around regular points, the free boundary is a Lipschitz $(n-1)$-dimensional manifold. 

Let us suppose 0 is a regular point. Take $\bar \tau = \frac{1}{\sqrt{2}}\left(\be_1 + \be_i\right)$ with $i\in \{2,\dots,n\}$, and notice that in $B_\rho$ such that \eqref{eq.lip0} holds (with $\theta_\circ= \pi/4$) we have that $v_1 := \de_{\be_1} u$ and $v_2 := \de_{\bar\tau} u$ are positive harmonic functions, vanishing continuously on $\Omega := \Lambda(u)\cap B_{\rho}$, by Proposition~\ref{prop.lipfb}. Thus, $v_1/v_2$ is H\"older continuous, which implies that $\de_{\be_i} u/\de_{\be_1} u$ is H\"older continuous, up to $\Lambda(u)$, in $B_\rho$.

We finish by noticing that, if we take $x\in \{x_{n+1} = 0\}$ such that $u(x) = t$, then $\nu(x)$ denotes the unit normal vector to the level set $\{u = t\}$ on the thin space,  where
\[
\nu_i(x) := \frac{\de_{\be_i} u}{|(\de_{\be_1} u,\dots,\de_{\be_n} u)|} = \frac{\de_{\be_i} u/\de_{\be_1} u}{\left(1+ \sum_{i= 2}^n (\de_{\be_i} u/\de_{\be_1} u)^2\right)^{1/2}}.
\]
Thus, $\nu = (\nu_1,\dots,\nu_n)$ is H\"older continuous. In particular, letting $t\downarrow 0$ we obtain that the normal vector to the free boundary is H\"older continuous, and therefore, the free boundary is $C^{1,\alpha}$ in $B_{\rho/2}$. 
\end{proof}

It is possible to keep iterating a higher order boundary Harnack principle to obtain higher order free boundary regularity estimates around regular points. Hence, Theorem~\ref{thm.bhp} also has a higher order analogy.

\begin{prop}[Higher order Boundary Harnack Principle, \cite{DS16}]
\label{prop.hobhp}
Let $\Omega\subset \{x_{n+1} = 0\}\cap B_1$ be a $C^{k,\alpha}$ domain on the thin space for $k\ge 1$, and let $v_1, v_2\in C(B_1)$ satisfying $\Delta v_1 = \Delta v_2 = 0$ in $B_1\setminus \Omega$. Assume, moreover, that $v_1$ and $v_2$ vanish continuously on $\Omega$, and $v_2 > 0$ in $B_1\setminus \Omega$. Then, $\frac{v_1}{v_2}$ is $C^{k,\alpha}$ in $B_{1/2}\setminus \Omega$ up to $\Omega$. 

Moreover, if $U_0(x') = \sqrt{{\rm dist}(x', \Omega)}$, and $v_1$ is even in $x_{n+1}$, then $\frac{v_1}{U_0}$ is $C^{k-1,\alpha}$ in $B'_{1/2}\setminus \Omega$ up to $\Omega$.
\end{prop}

And from the higher order Boundary Harnack Principle we can deduce higher order regularity of the free boundary (at regular points). 

\begin{cor}[$C^{\infty}$ regularity of the free boundary around regular points]
Let $u$ be a solution to \eqref{eq.TOP0}. Then, the set of regular points, ${\rm Reg}(u)$, is locally a $C^{\infty}$ $(n-1)$-dimensional manifold.  
\end{cor}
\begin{proof}
Follows analogously to the proof of Theorem~\ref{thm.C1a} by using Proposition~\ref{prop.hobhp} instead of Theorem~\ref{thm.bhp}.
\end{proof}

We refer to \cite{KPS15} for an alternative approach to higher regularity (that yields, in fact, that the free boundary is analytic).

As a consequence of the previous results we also get an expansion around regular points, proving that, up to lower order terms, the solution behaves like the half-space solution. In particular, this next theorem proves the uniqueness of blow-ups.

\begin{thm}[Expansion around regular points]
\label{thm.exp_reg}
Let $u$ be a solution to \eqref{eq.TOP0}, and let us assume $0\in {\rm Reg}(u)$. Then, there exists some $c > 0$ and some $\alpha > 0$ such that
\[
u(x) = c u_0(x) + o\left(|x|^{\frac32+\alpha}\right),
\]
where $u_0$ is the blow-up of $u$ at $0$ (i.e., $u_0(x) = {\rm Re}\left((x_1 + i |x_{n+1}|)^{3/2}\right)$ up to a rotation in the thin space).
\end{thm}
\begin{proof}
We here use the second part of Proposition~\ref{prop.hobhp}. By taking $\tau\in \mathbb{S}^{n}\cap \{x_{n+1} = 0\}$ and $v_1 = \partial_\tau u$ (a tangential derivative to the thin space), by Proposition~\ref{prop.hobhp} applied with $\Omega = \Lambda(u)$ (the contact set) we have
\[
\frac{\partial_\tau u}{U_0}\in C^\alpha
\]
in the thin space, for some fixed $\alpha > 0$ (coming from the boundary Harnack), outside of the contact set and up to the free boundary. In particular, 
\[
\left|\frac{\partial_\tau u}{U_0}(x') - c_0\,\right|  \le C|x'|^\alpha\quad\Longrightarrow\quad |\partial_\tau u(x') - c_0U_0(x')|\le CU_0(x')|x'|^{\alpha}\le C|x'|^{\frac12 +\alpha},
\]
for some constant $c_0 = \frac{\partial_\tau u}{U_0}(0)$. We recall that $U_0(x') = \sqrt{{\rm dist}(x', \Omega)}$. By the $C^{1,\alpha}$ re\-gu\-larity of the free boundary, there exists some $c_\tau$ such that $U_0 - c_\tau \partial_\tau u_0 = o\left(|x|^{\frac12 + \alpha'}\right)$ for some $\alpha' > 0$, where $u_0$ is the blow-up at 0. Thus, we have that
\[
|\partial_{\be_i} u(x') - c_i\partial_{\be_i}  u_0(x')|\le C|x'|^{\frac12 +\alpha'}.
\]
From the local uniform convergence $\partial_\tau u_r \to \partial_\tau u_0$ we must have $c_i = c\ge 0$ for all $i = 1,\dots,n$ in the previous expression,
where 
\[
c = \lim_{r\downarrow 0} r^{-\frac32} \left(\ave_{\de B_r} u^2\right)^{\frac12}.
\]

Thus, 
\[
|\nabla_{x'} u(x') - c\nabla_{x'} u_0(x')|\le C|x'|^{\frac12+ \alpha'}. 
\]
Since $\nabla_{x'}u(0) = \nabla_{x'}u_0(0) = 0$, by integrating the previous expression we deduce 
\[
|u(x') - c u_0(x')|\le C|x'|^{\frac32+ \alpha'}.
\]
By harmonic estimates, such inequality also holds outside of the thin space.
Now, if $ c =0$, it means that the frequency at 0 is at least $\frac32+\alpha'$. This contradicts 0 being a regular point, and thus, $c > 0$. This concludes the proof.
\end{proof}

We finish by noticing the uniqueness of blow-ups at regular points.

\begin{cor}[Uniqueness of blow-ups at regular points]
Let $u$ be a solution to \eqref{eq.TOP0}, and let us assume $0\in {\rm Reg}(u)$. Then, up to a rotation, 
\[
\frac{u(r\,\cdot)}{r^{\frac32}}\to c u_0\quad\text{as}\quad r\downarrow 0,
\]
locally uniformly, for some $c > 0$. Here, $u_0(x) = {\rm Re}\left((x_1 + i |x_{n+1}|)^{3/2}\right)$. 
\end{cor}
\begin{proof}
This is a direct consequence of Theorem~\ref{thm.exp_reg}.
\end{proof}

Notice that in the previous result we are adapting the blow-up sequence to the type of point we are dealing with. In particular, the corresponding blow-up obtained above is a constant multiple of the blow-up obtained with the sequence \eqref{eq.blowupseq}. 

\section{Singular points}
\label{sec.8}
In the classical (or thick) obstacle problem, all points of the free boundary have frequency 2, and thus the classification of free boundary points must be performed differently: regular points are those such that the contact set has positive density, whereas singular points are those where the contact set has zero density. 

This motivates the definition of singular point. Whereas it is not true that all points of positive density belong to the set ${\rm Reg}(u)$ as defined above, one can characterize the points with zero density. 

Let us start defining the set of singular points, which was originally studied by Garofalo and Petrosyan in \cite{GP09}. Let $u$ denote a solution to the thin obstacle problem, \eqref{eq.TOP0}, then we define
\begin{equation}
\label{eq.sing}
{\rm Sing}(u) := \left\{x\in \Gamma(u) : \liminf_{r\downarrow 0} \frac{\mathcal{H}^n(\Lambda(u)\cap B_r(x))}{\mathcal{H}^n(B_r(x)\cap\{x_{n+1} = 0\})} = 0 \right\},
\end{equation}
where we recall that $\Lambda(u)$ denotes the contact set, and $\mathcal{H}^n(E)$ denotes the $n$-dimensional Hausdorff measure of a set $E$. 

The first result in this direction involves the characterization of such points.

\begin{prop}[Characterization of singular points, \cite{GP09}]
Let $u$ be a solution to \eqref{eq.TOP0}. Then, the set of singular points \eqref{eq.sing} can be equivalently characterized by 
\[
{\rm Sing}(u) = \left\{x\in \Gamma(u) : N(0^+, u, x) = 2m,~~ m\in \N \right\}.
\]
That is, singular points are those with even frequency.
\end{prop}
\begin{proof}
Let us suppose that $0\in {\rm Sing}(u)$ according to definition \eqref{eq.sing}, and take a sequence $r_j\downarrow 0$ such that 
\begin{equation}
\label{eq.singlim}
\frac{\mathcal{H}^n(\Lambda(u)\cap B_{r_j})}{\mathcal{H}^n(B_{r_j}\cap\{x_{n+1} = 0\})} \to 0.
\end{equation}

Consider the sequence $u_{r_j}$, and after taking a subsequence if necessary, let us assume $u_{r_j}\to u_0$ uniformly in $B_1$. Notice that $\Delta u_{r_j} $ is a  non-positive measure supported on $\Lambda(u_{r_j})$. By assumption, $\mathcal{H}^n(\Lambda(u_{r_j})\cap B_1)\to 0$. Thus, since $u_{r_j}$ converges uniformly to $u_0$, $u_0$ has Laplacian concentrated on a set with zero harmonic capacity, and thus, it is harmonic. 

By Theorem~\ref{thm.blowupthm}, $u_0$  is a global  homogeneous solution to the thin obstacle problem, with homogeneity $\kappa := N(0^+, u)$. In particular, being homogeneous and harmonic, it must be a polynomial. Moreover, since $u_r$ is even with respect to $\{x_{n+1} = 0\}$, so is $u_0$. Thus, $u_0$ is a non-zero, harmonic polynomial, even with respect to $\{x_{n+1} = 0\}$ and non-negative on the thin space. Its homogeneity must be even, and thus $\kappa = 2m$ for some $m\in \N$. 

Suppose now that $0\in \Gamma(u)$ is such that $N(0^+, u) = 2m$ for some $m\in \N$. Take any blow-up of $u$ at zero, $u_0$. Then $u_0$ is a global solution to the thin obstacle problem, with homogeneity $2m$. As a consequence $u_0$ must be harmonic everywhere, and thus, an homogeneous harmonic polynomial (we refer to \cite[Lemma 7.6]{Mon09} or \cite[Lemma 1.3.4]{GP09} for a proof of this fact). 

Now, since $u_0$ is non-zero even homogeneous harmonic polynomial, and is non-zero on the thin space (by Cauchy-Kovalevskaya), $\mathcal{H}^n(\{u_0= 0\} \cap \{x_{n+1} = 0\}) = 0$. Thus, from the uniform convergence $u_{r_j} \to u_0$, we must have that \eqref{eq.singlim} holds. 
\end{proof}

Thus, the set of singular points consists of those points with even homogeneity. It is then natural to define 
\[
\Gamma_\lambda(u) := \{x\in \Gamma(u) : N(0^+, u, x) = \lambda\},
\]
so that 
\[
{\rm Sing}(u) = \bigcup_{m \in \N} \Gamma_{2m}(u) =: \Gamma_{\rm even}(u).
\]

In fact, singular points present a particularly good structure. At singular points of order $2m$, the solution to the thin obstacle problem is $2m$ times differentiable (in the sense of \eqref{eq.expu} below) and in particular, the blow-up is unique (see Remark~\ref{rem.typesofblowups} below), and belongs to the set 
\[
\mathcal{P}_{2m} := \{p : \Delta p = 0, x\cdot \nabla p = 2mp, p(x', 0) \ge 0, p(x', x_{n+1}) = p(x', -x_{n+1})\},
\]
$2m$-homogeneous, harmonic polynomials, non-negative on the thin space. 
That is, the following result from \cite{GP09}, which we will not prove, holds.

\begin{thm}[Uniqueness of blow-ups at singular points, \cite{GP09}]
\label{thm.kdiff}
Let $u$ be a solution to \eqref{eq.TOP0}. Let $x_\circ\in \Gamma_{2m}(u)$ for some $m\in \N$. Then, there exists a non-zero polynomial $p_{x_\circ}\in \mathcal{P}_{2m}$ such that 
\begin{equation}
\label{eq.expu}
u(x) = p_{x_\circ}(x-x_\circ) + o(|x-x_\circ|^{2m}). 
\end{equation}
In particular, the blow-up at 0 is unique. Moreover, the map $x_\circ\ni \Gamma_{2m}(u)\mapsto p_{x_\circ}$ is continuous. 
\end{thm}

\begin{rem}
\label{rem.typesofblowups}
In the proof of the previous result, one defines $p_{x_\circ}$ at a free boundary point $x_\circ$ of frequency $2m$ as 
\begin{equation}
\label{eq.blowupsing}
p_{x_\circ}(x) := \lim_{r\downarrow 0} \frac{u(x_\circ + rx)}{r^{2m}}.
\end{equation}
Then $p_{x_\circ}$ is well-defined, and it is sometimes called the \emph{blow-up} or \emph{first blow-up} at singular points. The blow-up $p_{x_\circ}$ is a (non-zero) multiple of the blow-up obtained by the sequence \eqref{eq.blowupseq}. 
From now on, when referring to the blow-up at a singular point we refer to the one obtained by \eqref{eq.blowupsing}, which is uniquely determined. 
\end{rem}

The proof of the previous theorem is based on Weiss and Monneau-type monotonicity formulas, saying that if $u$ has a singular points of order $2m$ at the origin, the following functions are non-decreasing,
\begin{equation}
\label{eq.w2m}
r\mapsto W_{2m} (r, u) := \frac{1}{r^{n-1+4m}}\int_{B_r}|\nabla u|^2 - \frac{2m}{r^{n+4m}}\int_{\partial B_r} u^2
\end{equation}
and (as a consequence),
\[
r\mapsto M_{2m}(r, u, p_{2m}) = \frac{1}{r^{n+2m}}\int_{\partial B_r} (u-p_{2m})^2,
\]
for all $p\in \mathcal{P}_{2m}$ and $0< r< 1$.
From here, in \cite{GP09} the authors establish first non-degeneracy at singular points (which, in particular, yields that \eqref{eq.blowupsing} is non-zero), and then the uniqueness of a blow-up. The continuity with respect to the point then follows by a compactness argument.

Theorem \ref{thm.kdiff} establishes a connection between singular points and their blow-ups. This also allows to separate between different singular points according to ``how big'' the contact set is around them. We already know it has zero $\mathcal{H}^n$-density. In fact, the contact set around singular points has the same ``size'' as the translation invariant set of the blow-up (see \eqref{eq.tis} below). Thus, we can establish a further stratification within the set of singular points, according to the size of the translation invariant set (which is a subspace) of the blow-up. 

Given a solution $u$ to the thin obstacle problem \eqref{eq.TOP0}, and given $x\in \Gamma(u)$, let us denote by $p_x$ any blow-up of $u$ at $x$. In particular, if $x$ is a singular free boundary point, $p_x\in \mathcal{P}_{2m}$ (defined as \eqref{eq.blowupsing} in Remark~\ref{rem.typesofblowups}) is uniquely determined by the result above.

Let us denote by $L(p)$ the translation invariant set for $p$, where $p$ is a blow-up,
\begin{equation}
\label{eq.tis}
\begin{split}
L(p) & := \left\{\xi \in \R^{n+1} : p(x+\xi) = p(x) \textrm{ for all } x\in \R^{n+1}\right\}\\
& = \left\{\xi \in \R^{n+1} : \xi \cdot \nabla p(x) = 0 \textrm{ for all } x\in \R^{n+1}\right\},
\end{split}
\end{equation}
where we recall that blow-ups $p$ are homogeneous. 
Then, if we denote 
\begin{equation}
\label{eq.gammaell}
\Gamma_{2m}^\ell := \{x\in \Gamma_{2m} : \dim{L(p_x)} = \ell\},\quad \ell\in\{0,\dots,n-1\},
\end{equation}
we have
\[
{\rm Sing}(u) = \Gamma_{\rm even}(u) = \bigcup_{m\in \N} \Gamma_{2m} = \bigcup_{m\in\N}\bigcup_{\ell = 0}^{n-1}\Gamma^\ell_{2m}.
\]

As a consequence of Theorem~\ref{thm.kdiff}, combined with Whitney's extension theorem and the implicit function theorem, one can prove the following result regarding the structure of the singular set.
\begin{thm}[\cite{GP09}]
\label{thm.gp09}
Let $u$ be a solution to \eqref{eq.TOP0}. Then, the set $\Gamma_{2m}^\ell(u)$ (see \eqref{eq.gammaell}) for $\ell\in \{0,\dots,n-1\}$, is contained in a countable union of $C^1$ $\ell$-dimension manifolds. 
\end{thm}

Notice that the fact that each stratum of the singular set is contained in countable union of manifolds (rather than  a single manifold) is unavoidable: there could be accumulation of lower-order points (say, of order 2) to higher order points (say, of order 4).  We refer the reader to the original papers, as well as \cite{PSU12, Sal12, DaSa18}, for the proofs of the previous statements.

On the other hand, the previous result can also be applied to the whole singular set: ${\rm Sing}(u)$ can be covered by a countable union of $C^1$ $(n-1)$-dimensional manifolds. The fact that the manifold is $C^1$ is due to the expansion of the solution \eqref{eq.expu}. In \cite{FJ18}, Jhaveri and the author show higher order expansions at singular points $x_\circ\in \Gamma_{2m}(u)$, analogous to \eqref{eq.expu}, as 
\begin{equation}
\label{eq.expuq}
u(x) = p_{x_\circ}(x-x_\circ) + q_{x_\circ}(x-x_\circ) + o(|x-x_\circ|^{2m+1})
\end{equation}
for some $(2m+1)$-homogeneous, harmonic polynomial $q_{x_\circ}$. Expansion of the form \eqref{eq.expuq} hold at almost every singular point, and thus, analogously to the previous case we obtain a structure result, that holds for all singular points up to a lower dimensional set: 
\begin{thm}[\cite{FJ18}]
\label{thm.fj18}
Let $u$ be a solution to \eqref{eq.TOP0}. Then, there exists a set $E\subset {\rm Sing}(u)$ of Hausdorff dimension at most $n-2$ such that ${\rm Sing}(u) \setminus E$ is contained in a countable union of $C^2$ $(n-1)$-dimensional manifolds.
\end{thm}

\subsection{The non-degenerate case}
So far we have been studying the thin obstacle problem with zero obstacle. When solving for an (even) boundary datum 
\[
g\in C^0(\de B_1), \quad g(x', x_{n+1}) = g(x', -x_{n+1})
\]
the problem looks like
\begin{equation}
\label{eq.TOP0_2}
  \left\{ \begin{array}{rcll}
  u & \ge & 0 & \textrm{ on }B_1 \cap \{x_{n+1} = 0\}\\
  \Delta u&=&0 & \textrm{ in } B_1\setminus\left(\{x_{n+1} = 0\}\cap \{u = 0\}\right)\\
  \Delta u & \le & 0 & \textrm{ in } B_1\\
    u& = & g & \textrm{ on } \de B_1,
  \end{array}\right.
\end{equation}

We had reduced to this problem from \eqref{eq.thinobst_intro_2} by subtracting the harmonic even extension of the analytic obstacle $\varphi$. Alternatively, from \eqref{eq.TOP0_2} we can reduce to the case of zero boundary data by subtracting the harmonic extension of $g$ to the unit ball. Thus, we obtain a problem of the form 
\begin{equation}
\label{eq.TOP0_3}
  \left\{ \begin{array}{rcll}
  v & \ge & \varphi & \textrm{ on }B_1 \cap \{x_{n+1} = 0\}\\
  \Delta v&=&0 & \textrm{ in } B_1\setminus\left(\{x_{n+1} = 0\}\cap \{v = \varphi\}\right)\\
  \Delta v & \le & 0 & \textrm{ in } B_1\\
    v& = & 0 & \textrm{ on } \de B_1,
  \end{array}\right.
\end{equation}
that is, a thin obstacle problem with obstacle $\varphi$. Problems \eqref{eq.TOP0_2} and \eqref{eq.TOP0_3} are the same when
\begin{equation}
\label{eq.transf}
\left\{
\begin{array}{ll}
\Delta \varphi = 0& \textrm{ in } B_1\\
\varphi = -g& \textrm{ on } \de B_1.
\end{array}
\right.
\qquad 
\textrm{and }\qquad 
v = u  + \varphi.
\end{equation}

In this setting, we say that problem \eqref{eq.TOP0_3} with $\varphi \in C^{3, 1} (B_1\cap\{x_{n+1} = 0\})$ is non-degenerate if 
\begin{equation}
\label{eq.nondeg}
\Delta_{x'} \varphi \le -c_\circ < 0\textrm{  in $B_1\cap \{x_{n+1} = 0\}\cap\{\varphi > 0\}$},\quad \varnothing\neq \{\varphi > 0\},
\end{equation}
where $\Delta_{x'}$ denotes the Laplacian in the first $n$ coordinates (Laplacian along the thin space). The last condition above is to avoid having a non-active obstacle. Alternatively, in terms of problem \eqref{eq.TOP0_2} we have 
\begin{equation}
\textrm{\eqref{eq.TOP0_2} is non-degenerate } \stackrel{\text{def.}}{\Longleftrightarrow} \varphi_g :  \left\{
\begin{array}{ll}
\Delta \varphi_g = 0& \textrm{ in } B_1\\
\varphi_g = -g& \textrm{ on } \de B_1.
\end{array}
\right. \textrm{satisfies \eqref{eq.nondeg}}.
\end{equation}

In particular, when we deal with concave obstacles, our problem is non-degenerate. In \cite{BFR18}, Barrios, Figalli, and Ros-Oton show that, under this non-degeneracy assumption, we have a better characterization of free boundary points. 
\begin{thm}[\cite{BFR18}]
Let $u$ be a solution to \eqref{eq.TOP0_2}, and suppose that the non-degeneracy condition \eqref{eq.nondeg} holds. Then, there exists a constant $\bar c$ (depending on $c_\circ$) such that for any $x_\circ\in \Gamma(u)\cap B_{1/2}$,
\[
\sup_{B_r (x_\circ)} u \ge \bar c\,r^2,
\]
for all $r\in (0,\textstyle{\frac14})$. 
In particular, if \eqref{eq.nondeg} holds, then
\[
\Gamma(u) = {\rm Reg}(u)\cup \Gamma_2(u),
\]
i.e., the free boundary consists only of regular points and singular points of order 2. 
\end{thm}
\begin{proof}
We prove the result for $v$ satisfying \eqref{eq.TOP0_3} and the proof follows by the transformation \eqref{eq.transf} with $\varphi = \varphi_g$ as in \eqref{eq.nondeg}. 

Let us define for $\bar x  = (\bar x', 0) \in B_{1/2}\cap \{x_{n+1} = 0\}\cap \{v > \varphi\}$,
\[
w_{\bar x}(x', x_{n+1})  = v(x', x_{n+1}) - \varphi(x') - \frac{c_\circ}{2n + 2}\left(|x'-\bar x'|^2 + x_{n+1}^2\right),
\]
where $c_\circ$ is the constant in \eqref{eq.nondeg}. Notice that, since $\Delta v = 0$  outside of the contact set $\Lambda(v)$,
\[
\Delta w_{\bar x} = - \Delta_{x'} \varphi - c_\circ \ge 0,\quad\textrm{in}\quad B_r(\bar x)\setminus \Lambda(v). 
\]
On the other hand, $w_{\bar x}(\bar x', 0) > 0$ and $w < 0$ on $\Lambda(v)$. By maximum principle, we must have 
$
\sup_{\de B_r(\bar x)} w_{\bar x} > 0.
$ Letting $\bar x \to x_\circ\in \Gamma(u)$ we deduce 
\[
\sup_{\de B_r(x_\circ)} w_{x_\circ} \ge 0,
\]
which implies the desired result. 

Finally, since the growth at the free boundary is at least quadratic, there cannot be any blow-up at a free boundary point with homogeneity greater than 2. 
\end{proof}

In this case, therefore, the non-regular part of the free boundary consists exclusively of singular points of order 2. In particular, in Theorem~\ref{thm.gp09} we have instead a single $C^1$ $\ell$-dimensional manifold covering (locally) the whole of $\Gamma_{2}^{\ell}(u)$. We can also establish a more refined version of Theorem~\ref{thm.fj18}, 

\begin{thm}[\cite{FJ18}]
\label{thm.fj18_2}
Let $u$ be a solution to \eqref{eq.TOP0}, and suppose that the non-degenerate condition \eqref{eq.nondeg} holds. Then, 
\begin{enumerate}[label=(\roman*)]
\item $\Gamma_2^{0}(u)$ is isolated in ${\rm Sing}(u) = \Gamma_2^0(u) \cup \cdots \cup \Gamma_2^{n-1}(u)$.
\item There exists an at most countable set $E_{1} \subset \Gamma_2^1(u)$ such that $\Gamma_2^1(u) \setminus E_{1}$ is locally contained in a single one-dimensional $C^{2}$ manifold. 
\item For each $\ell \in \{2,\dots,n-1\}$, there exists a set $E_m \subset \Gamma_2^\ell(u)$ of Hausdorff dimension at most $\ell-1$ such that $\Gamma_2^\ell(u) \setminus E_\ell$ is locally contained in a single $\ell$-dimensional $C^{2}$ manifold.
\end{enumerate}
\end{thm}

\subsection{An alternative approach to regularity: the epiperimetric inequality}
\label{ssec.epi}
A mo\-dern and prolific tool to tackle regularity questions for the free boundary in the thin obstacle problem has been the use of epiperimetric inequalities. Let us very briefly introduce the concept and some interesting consequences. In particular, we present a very recent result in which this technique is used to improve Theorem~\ref{thm.gp09} to an explicit modulus of continuity for the normal to the manifold containing singular points. 

In order to study singular points (and deduce Theorem~\ref{thm.gp09}) Garofalo and Petrosyan in \cite{GP09} prove the monotonicity of the Weiss energy along sequences of blow-ups, thus generalizing the Weiss monotonicity formula originally introduced in the classical obstacle problem (see \cite{Wei99}). In particular, they show that if $0$ is a free boundary point of order $\lambda$ (cf. \eqref{eq.w2m} where only even frequencies were involved) and $u$ is a solution to the thin obstacle problem \eqref{eq.TOP0}, then the \emph{Weiss' boundary adjusted energy}
\[
W_\lambda (r, u) := \frac{1}{r^{n-1+2\lambda}}\int_{B_r}|\nabla u|^2 - \frac{\lambda}{r^{n+2\lambda}}\int_{\partial B_r} u^2
\]
satisfies that $r\mapsto W_\lambda(r, u)$ is non-decreasing for $r \in (0, 1)$. More precisely, we have that 
\[
\frac{d}{dr} W_\lambda(r, u) = \frac{2}{r^{n+1+2\lambda}}\int_{\partial B_r} (x\cdot \nabla u - \lambda u)^2.
\]

The Weiss energy at scale $r$ is related to the speed of convergence of a solution to its blow-up (or, as seen in the previous expression, the closeness to a $\lambda$-homogeneous function), and thus, it is not surprising that a detailed study of such energy could lead to a better understanding of free boundary points.

A quantification of the value of $W_\lambda(r, u)$ (the type of convergence as $r\downarrow 0$) leads to a better understanding of free boundary points of that given frequency. Such quantification can be made by means of epiperimetric inequalities. 

For example, at regular points (namely, when the blow-up is $\frac32$-homogeneous), one can prove that, for some (explicit) dimensional constant $\kappa$, if $c\in H^1(B_1)$ and $z$ is $\frac32$-homogeneous, $z = c$ on $\partial B_1$, even (in $x_{n+1}$) and non-negative on the thin space, then there exists some $v$ even and non-negative on the thin space with $v = c$ on $\partial B_1$ such that
\[
W_{\frac32}(1, v) \le (1-\kappa) W_{\frac32}(1, z),
\]
namely, we have an explicit improvement of homogeneity\footnote{The epiperimetric inequality can be presented in different statements. The statement presented here is the one by Colombo, Spolaor, and Velichkov, \cite{CSV19}, and similar statements with extra assumptions can also be found on \cite{FS10, GPS16}.}. This is called an epiperimetric inequality, and from here the convergence of $W_\lambda(r, u)$ can be quantified (in this case, it decays like a power as $r\downarrow 0$) and one can deduce the free boundary regularity at regular points.
(See \cite{FS10, GPS16, CSV19}.)   This approach is the one used, for example, in \cite{GPS16, GP16} in order to prove regularity of the free boundary at regular points for the variable coefficient thin obstacle problem, and for the fractional obstacle problem with subcritical drift, respectively; and even in \cite{Shi18} in the context of the parabolic version of the thin obstacle problem. 

In \cite{CSV19}, Colombo, Spolaor, and Velichkov are able to establish, for the first time, an epiperimetric inequality that is valid at all singular points. In particular, they establish a logarithmic-type epiperimetric inequality at even frequency points: under the previous assumptions on $c$ and $z$ but with homogeneity $2m$ (instead of $\frac32$), there exists some $h$ even, non-negative on the thin space, and with $h = c$ on $\partial B_1$ such that 
\[
W_{2m}(1, h) \le W_{2s}(1, z) (1-\kappa_m|W_{2m}(z)|^\gamma),
\]
where $\kappa_m$ is a dimensional constant that depends also on $m$, and $\gamma = \frac{n-1}{n+1}$. This is called a logarithmic-type epiperimetric inequality, since it can only lead to logarithmic decay for the Weiss energy. The proof of this result in \cite{CSV19} is done by a direct method: the authors are able to construct a competitor $h$ from the Fourier decomposition of the trace of $c$ (or $z$). 

From here, they quantify the convergence of the rescaling of $u$ to its blow-up (with a logarithm), and as a consequence, they are able to establish an explicit modulus of continuity for the map $x_\circ\ni \Gamma_{2m}(u) \mapsto p_{x_\circ}$ appearing in Theorem~\ref{thm.kdiff}.  Proceding with the Whitney's extension theorem as in the proof of Theorem~\ref{thm.gp09} they are able to establish the $C^{1,\log^{-\epsilon_\circ}}$ regularity\footnote{We say that a manifold is $C^{1,\log^{-\epsilon_\circ}}$ if the normal derivative has as modulus of continuity (up to a constant) $\sigma(t) = \log^{-\eps_\circ}(1/t)$ for some $\eps_\circ > 0$ fixed.} of the covering manifolds:

\begin{thm}[\cite{CSV19}]
\label{thm.CSV19}
Let $u$ be a solution to \eqref{eq.TOP0}. Then, the set $\Gamma_{2m}^\ell(u)$ (see \eqref{eq.gammaell}) for $\ell\in \{0,\dots,n-1\}$, is contained in a countable union of $C^{1,\log^{-\epsilon_\circ}}$ $\ell$-dimension manifolds. 
\end{thm}

Their use of the logarithmic-type epiperimetric inequality does not end here, see Theorem~\ref{thm.epiCSV} below for another interesting consequence of their result.

\section{Other points}
\label{sec.9}
The free boundary contains, in general, other points different from \emph{regular} and \emph{singular}. Even in two dimensions ($n = 1$) one can perform the simple task of manually classifying all the possible homogeneities that an homogeneous solution to the thin obstacle problem (with zero obstacle) can present (see \cite[Proposition A.1]{FS18}) .  

Indeed, for $n = 1$ homogeneous solutions to the thin obstacle problem must have homogeneity belonging to the set 
\[
\left\{2m, 2m-\frac12, 2m+1\right\}_{m\in \N}. 
\]

Solutions with homogeneity $2m$ are harmonic polynomials, non-negative on the thin space. On the other hand, homogeneous solutions with homogeneity $2m-\frac12$ or $2m+1$ are of the form 
\[
{\rm Re}\left((x_1 + i|x_2|)^{2m-\frac12}\right)\qquad\textrm{and}\qquad {\rm Im}\left((x_1 + i|x_2|)^{2m+1}\right),\quad \textrm{for $m \in \N$}.
\]
Notice that when the homogeneity is $2m-\frac12$ we have \emph{half-space} solutions on the thin space. Indeed, in this case, solutions restricted to $x_2 = 0$ are of the form $u(x_1, 0) = (x_1)_+^{2m-1/2}$. %On the other hand, solutions with odd homogeneity are identically zero on the thin space (in particular, this type of homogeneous solution is \emph{not} an example of a free boundary point with odd homogeneity, and in fact, they do not exist in dimension $n = 1$). 

Given that no other homogeneities can appear in dimension 2, one can show that, in any dimension, the previous homogeneities comprise all of the free boundary, up to a lower dimensional set. It is for this reason that we separate the possible homogeneities of the free boundary as
\begin{equation}
\label{eq.cffb}
\Gamma(u) = \Gamma_{3/2}(u) \cup \Gamma_{\rm even}(u) \cup \Gamma_{\rm odd}(u) \cup \Gamma_{\rm half}(u) \cup \Gamma_*(u), 
\end{equation}
where $\Gamma_{3/2}(u) = {\rm Reg}(u)$ are regular points;  $\Gamma_{\rm even}(u) = {\rm Sing}(u)$ are singular points; $\Gamma_{\rm odd}(u)$ denotes the set of points with odd homogeneity $2m+1$ for $m\in \N$;  $\Gamma_{\rm half}(u)$ are the points with homogeneity $2m+\frac32$ for $m \in \N$; and $\Gamma_*(u)$ are the rest of possible free boundary points (in particular, $\Gamma_*(u) = \varnothing$ if $n=1$, and we will see that ${\rm dim}_{\mathcal{H}}(\Gamma_*(u))\le n-2$ in general). 

\subsection{The set $\Gamma_{\rm odd}(u)$} The free boundary points belonging to $\Gamma_{\rm odd}(u)$ are those with odd homogeneity $2m+1$ for $m\in \N$. They are analogous to the singular set, in the sense that in this case, points belonging to $\Gamma_{\rm odd}(u)$ can also be characterized via the density of the contact set: these points have density 1. 

They are not known to exist (no single example has been constructed so far). Notice that the homogeneous solutions presented above are vanishing identically on the thin space, and thus they do not have a free boundary. 

In fact, in dimension $n = 1$, if such a point existed its blow-up would be of the form 
\begin{equation}
\label{eq.imform}
{\rm Im}\left((x_1 + i|x_2|)^{2m+1}\right),\quad \textrm{for $m \in \N$}.
\end{equation}
(Think, for example, of the $x_2$-even extension of the harmonic polynomial $x_2^3-3x_1^2x_2 $ for $x_2 \ge 0$.) However, solutions of the form \eqref{eq.imform} have non-vanishing normal derivative on the thin space outside of the origin, whereas a free boundary point can be approximated by points with vanishing normal derivative. By finding the blow-up along the sequence of points with vanishing normal derivative, from the $C^1$ convergence of blow-ups we reach a contradiction: free boundary points with odd homogeneity do not exist in dimension $n = 1$. 

The set of points belonging to $\Gamma_{\rm odd}(u)$ has been studied in a recent work by Figalli, Ros-Oton, and Serra \cite[Appendix B]{FRS19}, and previously in \cite{FS18}.

\begin{prop}[Characterization of points in $\Gamma_{\rm odd}(u)$]
\label{prop.odd}
Let $u$ be a solution to \eqref{eq.TOP0}. Then, the set of points with odd homogeneity, $\Gamma_{\rm odd}(u)$, can be equivalently characterized by 
\begin{equation}
\label{eq.oddch}
\Gamma_{\rm odd}(u) := \left\{x\in \Gamma(u) : \limsup_{r\downarrow 0} \frac{\mathcal{H}^n(\Lambda(u)\cap B_r(x))}{\mathcal{H}^n(B_r(x)\cap\{x_{n+1} = 0\})} = 1 \right\}.
\end{equation}
That is, points with odd homegeneity are those where the contact set has density 1. 
\end{prop}
\begin{proof}
Let us suppose that $0\in \Gamma(u)$ fulfills definition \eqref{eq.oddch}, that is, we can take a sequence $r_j\downarrow 0$ such that 
\begin{equation}
\label{eq.singlim_2}
\frac{\mathcal{H}^n(\Lambda(u)\cap B_{r_j})}{\mathcal{H}^n(B_{r_j}\cap\{x_{n+1} = 0\})} \to 1.
\end{equation}

Consider the sequence $u_{r_j}$ (defined by \eqref{eq.blowupseq}), and after taking a subsequence if necessary, let us assume $u_{r_j}\to u_0$ uniformly in $B_1$. In particular, $u_0$ vanishes identically on the thin space. Since it is homogeneous, and harmonic on $x_{n+1} > 0$, it must be a polynomial. It cannot have even homogeneity, since by the discussion on singular points it would have zero density. Thus, it is an homogeneous harmonic polynomial with odd homogeneity in $x_{n+1} \ge 0$ (extended evenly in the whole space). Notice also that it cannot be linear (on each side) because the minimum possible homogeneity is $\frac32$.

On the other hand, suppose that $0\in \Gamma(u)$ is such that $N(0^+, u) = 2m+1$ for some $m\in \N$. Take any blow-up of $u$ at zero, $u_0$. Then $u_0$ is a global solution to the thin obstacle problem, with homogeneity $2m+1$. Let us define the global (homogeneous) solution to the thin obstacle problem given by $P$,
\[
P(x) = \sum_{j = 1}^n {\rm Im}\big((x_j+i|x_{n+1}|\big)^{2m+1},
\]
so that $\de_{x_{n+1}}^+ P < 0$ in $\{x_{n+1} = 0\}\setminus\{0\}$. Using \eqref{eq.delta_xn}, we obtain that for any test function $\Psi = \Psi(|x|)$ (so that $\nabla\Psi = \Psi'(|x|)\frac{x}{|x|}$),
\begin{align*}
2\int_{\{x_{n+1} = 0\}} \de_{x_{n+1}}^+ P \,\Psi u_0 & = \int \Delta P \Psi u_0 = -\int \left( \nabla P \cdot \nabla u_0 \Psi+\nabla P \cdot\nabla\Psi u_0\right) \\
& =  \int \left( P \Delta u_0 \Psi+P\nabla u_0 \cdot\nabla\Psi -u_0 \nabla P \cdot\nabla\Psi \right)\\
& = \int \left( P\nabla u_0\cdot x \frac{\Psi'(|x|)}{|x|}- u_0\nabla P\cdot x \frac{\Psi'(|x|)}{|x|}\right) = 0,
\end{align*}
where we have used that $P \Delta u_0 \equiv 0$ everywhere, and $\nabla u_0 \cdot x = (2m+1) u_0$, $\nabla P\cdot x = (2m+1) P$. Since $u_0\ge 0$ on the thin space, and $\de_{x_{n+1}}^+ P < 0$ outside of the origin on the thin space, we deduce $u_0\equiv 0$ on the thin space. 

As a consequence $u_0$ must be harmonic everywhere, vanishing on the thin space. Thus, it is an homogeneous harmonic polynomial with degree $2m+1$. In particular, $\de_{x_{n+1}} u_0$ is a non-zero $2m$-homogeneous polynomial on $\R^{n+1}_+$. From the $C^1$ convergence of $u_{r_j}\to u_0$ (that is, the uniform convergence of $\de_{x_{n+1}} u_{r_j}$ to $\de_{x_{n+1}} u_0$) we deduce \eqref{eq.oddch}. 
\end{proof}

We also have a result analogous to Theorem~\ref{thm.kdiff} at odd-frequency points. Let us start by defining for $m\ge 1$
\begin{align*}
\mathcal{Q}_{2m+1} := \big\{\,& \text{$q$ solution to the thin obstacle problem \eqref{eq.TOP0} in $\R^{n+1}$},\\
& x\cdot \nabla q = (2m+1)q,\,q(x', x_{n+1}) = q(x', -x_{n+1})\big\},
\end{align*}
namely, the set of $(2m+1)$-homogeneous even solutions to the thin obstacle problem (notice that by the proof of Proposition~\ref{prop.odd}, in particular, $q(x', 0) \equiv 0$). Then, we have 
\begin{thm}[Uniqueness of blow-ups at odd-frequency points, \cite{FRS19}]
\label{thm.uniqoddfreq}
Let $u$ be a solution to \eqref{eq.TOP0}. Let $x_\circ\in \Gamma_{2m+1}(u)$ for some $m\in \N$. Then, there exists a non-zero $q_{x_\circ}\in \mathcal{Q}_{2m+1}$ such that 
\begin{equation}
\label{eq.expu_o}
u(x) = q_{x_\circ}(x-x_\circ) + o(|x-x_\circ|^{2m+1}). 
\end{equation}
In particular, the blow-up at 0 is unique. Moreover, the set $\Gamma_{2m+1}(u)$ is $(n-1)$-rectifiable.
\end{thm}

The $(n-1)$-rectifiability of the set $\Gamma_{2m+1}(u)$ had been already proved in \cite{KW13, FS18}, see Theorem~\ref{thm.fs18} below.

\begin{rem}
\label{rem.typesofblowups2}
In order to establish this result, one needs to show first a non-degeneracy property around points belonging to $\Gamma_{2m+1}(u)$ analogous to the one in the case of singular points. As a consequence, one can then define 
\[q_{x_\circ}(x) := \lim_{r\downarrow 0} \frac{u(x_\circ + rx)}{r^{2m+1}},
\]
which is a non-zero element of $\mathcal{Q}_{2m+1}$. As in Remark~\ref{rem.typesofblowups}, $q_{x_\circ}$ is a (non-zero) multiple of the blow-up obtained by the sequence \eqref{eq.blowupseq}, and one can then define $q_x$ to be the \emph{blow-up} or \emph{first blow-up} at points $x\in \Gamma_{2m+1}(u)$.
\end{rem}

\subsection{The set $\Gamma_{\rm half}(u)$} The free boundary points belonging to $\Gamma_{\rm half}(u)$ are those with homogeneity $2m+\frac32$ for $m\in \N$.

They do exist: the homogeneous solutions are themselves examples of solutions to the thin obstacle problem with free boundary points belonging to $\Gamma_{\rm half}(u)$. Whereas they are currently not completely understood, they seem to exhibit a similar behaviour to regular points. However, the fact that they are not an open set (in the free boundary), makes it harder to study regularity properties of the free boundary around them.

There are not many results for points belonging to $\Gamma_{\rm half}$. The following proposition shows that points in $\Gamma_{\rm half}(u)$ can present a behaviour similar to that of regular points, but the converse is not true: we still do not know whether the set $\Gamma_{\rm half}$ is regular. 

\begin{prop}[\cite{FR19}]
Given a $C^\infty$ domain $\Omega\subset B_1\cap\{x_{n+1} = 0\}$, and $m\in \N$, there exists $\varphi\in C^\infty$, and $g\in C^0(\de B_1)$, such that the solution $u $ to the thin obstacle problem \eqref{eq.thinobst_intro_3} with obstacle $\varphi$ and boundary data $g$ has contact set  $\Lambda(u) = \Omega$, and all the points of the free boundary $\Gamma(u)$ have frequency $2m+\frac32$. 
\end{prop}

The proof of this proposition is an explicit construction based on a previous result by Grubb, \cite{Gru15}. 

On the other hand, epiperimetric inequalities (see subsection~\ref{ssec.epi}) have been used to get some regularity properties for this set for $n = 1$: in \cite{CSV19} the authors prove a classical epiperimetric inequality at points in $\Gamma_{\rm half}(u)$ in dimension $n = 1$, which gives $C^{1,\alpha}$ decay of the solution and uniqueness of blow-ups in this case. 

Finally, the only general result that establishes some regularity for the whole of the free boundary in the thin space (and, in particular, for $\Gamma_{\rm half}(u)$) is the following recent result by Focardi and Spadaro \cite{FS18} (based on the general approach introduced by Naber and Valtorta in \cite{NV17, NV15}), which shows the $\mathcal{H}^{n-1}$-rectifiability of the free boundary for the thin obstacle problem. As a consequence, the set $\Gamma_{\rm half}(u)$ is always contained in the countable union of $C^1$ manifolds, up to a set of zero $(n-1)$-dimensional Hausdorff measure. The same result had also been proved by Krummel and Wickramasekera in \cite{KW13} in the context of two-valued harmonic functions\footnote{After a transformation, the thin obstacle problem is a particular case of two-valued harmonic function. Therefore, rectifiability of the free boundary for the thin obstacle problem can also be deduced from \cite{KW13}}. 

\begin{thm}[\cite{KW13, FS18}]
\label{thm.fs18}
Let $u$ be a solution to \eqref{eq.TOP0}. Then the free boundary $\Gamma(u)$ is $(n-1)$-rectifiable. That is, there exists at most countably many $C^1$ $(n-1)$-dimensional manifolds $M_i$ such that 
\[
\mathcal{H}^{n-1}\left(\Gamma(u) \setminus \bigcup_{i\in \N} M_i\right) = 0,
\]
where $\mathcal{H}^{n-1}$ denotes the $(n-1)$-dimensional Hausdorff measure. 
\end{thm}

\subsection{The set $\Gamma_*(u)$}

We call $\Gamma_*(u)$ the rest of free boundary points. That is, points with homogeneity not belonging to the set $\{2m , 2m+1, 2m-\frac12\}_{m\in \N}$,
\begin{equation}
\label{eq.gammastar}
\Gamma_*(u) := \left\{x_\circ\in \Gamma(u) : N(0^+, u, x_\circ) \in (2, \infty)\setminus \bigcup_{m\in \N} \left\{2m, 2m+1, 2m-\frac12\right\}\right\}.
\end{equation}

It is currently not known whether such points exist, that is, which of the frequencies from the set $\left\{\frac32\right\}\cup [2, \infty)$ are admissible; although it is conjectured that the set $\Gamma_*(u)$ is empty, at least in low dimensions. 

The only result in this direction comes from the study of the epiperimetric inequality by Colombo, Spolaor, and Velichkov, introduced in subsection~\ref{ssec.epi}. In particular, they are able to show, by means of their logarithmic-type epiperimetric inequality at singular points, that even frequencies are isolated in the set of admissible frequencies: points with order \emph{close} to $2m$ do not exist (except for singular points themselves), where this closeness can be quantified with an explicit constant.

 \begin{thm}[\cite{CSV19}]
 \label{thm.epiCSV}
Let $u$ be a solution to the thin obstacle problem with zero obstacle,
 \begin{equation}
\label{eq.top_b}
  \left\{ \begin{array}{rcll}
  u & \ge & 0 & \textrm{ on }B_1 \cap \{x_{n+1} = 0\}\\
  \Delta u&=&0 & \textrm{ in } B_1\setminus\left(\{x_{n+1} = 0\}\cap \{u = 0\}\right)\\
  \Delta u & \le & 0 & \textrm{ in } B_1\\
    u& = & g & \textrm{ on } \de B_1,
  \end{array}\right.
\end{equation}
Let $\Gamma_\lambda(u)$ denote the points of frequency $\lambda > 0$. Then, 
\[
\Gamma_\lambda(u) = \varnothing \qquad\text{for every}\quad \lambda \in \bigcup_{m\in \N} \big((2m-c_m, 2m+c_m)\setminus\{2m\}\big),
\]
for some constants $c_m>0$ depending only on $m$ and $n$. 
 \end{thm}
 
Understanding whether epiperimetric inequalities exist at other frequencies, and in particular, establishing quantitative epiperimetric inequalities are other frequencies, can help in determining whether they are admissible or not.

The goal of the rest of the subsection is to prove that, if the set $\Gamma_*(u)$ exists, then it is lower dimensional. That is, we will show the following proposition, stating that points of frequency $\kappa\in(2, \infty)\setminus\{2m, 2m+1, 2m+\frac32\}_{m\in \N}$ are $n-2$ dimensional for all solutions to the thin obstacle problem, which corresponds to \cite[Theorem 1.3]{FS18}. We provide, however, a different proof, by means of a dimension reduction argument due to White, \cite{Whi97}. 

\begin{prop}
\label{prop.smallother}
Let $u$ be a solution to the thin obstacle problem with zero obstacle, \eqref{eq.top_b}.
Let us define  $\Gamma_*(u)\subset \Gamma(u)$ by \eqref{eq.gammastar}. 
Then
\[
\dim_{\mathcal{H}}{\Gamma}_*(u) \le n-2.
\]
Moreover, if $n = 2$, ${\Gamma}_*(u) $ is discrete.
\end{prop}

In this proposition, $\dim_{\mathcal{H}}$ denotes the Hausdorff dimension of  a set (see \cite{Mat95}). 

In order to prove this result, we will need two lemmas. We will use the notation $u^{x_\circ}(x)$ for $x_\circ \in \Gamma(u)$ to denote translations. That is, we denote
\[
u^{x_\circ}(x) = u(x'+x_\circ', x_{n+1}),
\]
so that, in particular, $N(r, u, x_\circ) = N(r, u^{x_\circ})$. 

\begin{lem}
\label{lem.epsclose}
Let $u$ be a solution to the thin obstacle problem \eqref{eq.top_b}. Let ${\Gamma}_*(u)$ be as in \eqref{eq.gammastar}.

Let $y_\circ\in {\Gamma}_*(u)$. Then,  for every $\eps>0$ there exists some $\delta > 0$ such that for every $\rho\in (0, \delta]$, there exists an $(n-2)$-dimensional linear subspace $L_{y_\circ, \rho}$ of $\R^n\times\{0\}$ such that
\[
\big\{x\in B_\rho(y_\circ)\cap \{x_{n+1} = 0\} : N(0^+,  u^x) \ge N(0^+,  u^{y_\circ}) - \delta\big\}\subset \{x : {\rm dist}(x, y_\circ + L_{y_\circ, \rho}) < \eps \rho\}.
\]
\end{lem}
\begin{proof}
Let us denote $\eta = N(0^+,  u^{y_\circ})\in(2, \infty)\setminus \{2m, 2m+1, 2m+\frac32\}_{m\in \N}$. Let us proceed by contradiction. Suppose that there exist $\eps > 0$, and sequences $\delta_k \downarrow 0$ and $\rho_k\downarrow 0$ such that
\begin{equation}
\label{eq.contr_a}
\{x\in B_{\rho_k}(y_\circ)\cap \{x_{n+1} = 0\} : N(0^+,  u^x) \ge \eta - \delta_k\}\not\subset \{x : {\rm dist}(x, y_\circ + L) < \eps \rho_k\}
\end{equation}
for every $(n-2)$-dimensional linear subspace $L$ of $\R^n\times\{0\}$.

In particular, if we denote $ u^{y_\circ}_r =  u^{y_\circ}(r\,\cdot )$ and $d_r = r^{-n/2}\|u^{x_\circ}\|_{L^2(\de B_r)}$, then $ u^{y_\circ}_{\rho_k} /d_{\rho_k}$ converges, up to subsequences, to some $v_\circ$ a global solution to the thin obstacle problem with zero obstacle, homogeneous of degree $\eta$. Let us denote $L(v_\circ)$ the invariant set in $\R^n\times\{0\}$ of $v_\circ$. In particular, it is a subspace of dimension at most $n-2$ (this follows since two dimensional homogeneous solutions to the thin obstacle problem have homogeneity belonging to $\{2m, 2m +1, 2m-\frac12\}_{m\in \N}$). As an abuse of notation, let us take as $L(v_\circ)$ any $(n-2)$-dimensional plane containing the invariant set.

Now, by assumption \eqref{eq.contr_a} and choosing $L = L(v_\circ)$, for every $k\in \N$ there exists some $x_k\in B_{\rho_k}(y_\circ)\cap\{x_{n+1} = 0\}$ with $N(0^+,  u^{x_k}) \ge \eta - \delta_k$ such that ${\rm dist}(x_k, y_\circ + L(v_\circ)) \ge \eps \rho_k$.

Let us denote $ z_k = \rho_k^{-1} (x_k-y_\circ)\in B_1(0)$, and notice that ${\rm dist}(z_k,  L(v_\circ)) \ge \eps $. By scaling, we know that
\[
N(0^+,  u^{x_k}) = N(0^+, u^{y_\circ}_{\rho_k}(\,\cdot + z_k)).
\]
Moreover, 
\[
d_{\rho_k}^{-1} u^{y_\circ}_{\rho_k} \to v_\circ \quad\textrm{uniformly in compact sets as }k\to \infty.
\]

Thus,
\[
\eta - \delta_k \le N(0^+,  u^{x_k}) = N(0^+, u^{y_\circ}_{\rho_k}(\,\cdot + z_k)) = N(0^+, d_{\rho_k}^{-1} u^{y_\circ}_{\rho_k}(\,\cdot + z_k)),
\]
and by the upper semi-continuity of the frequency function (and after taking a subsequence such that $z_k\to z\in B_1(0)$) we get that
\[
N(0^+, v_\circ(\,\cdot + z))\ge \eta,
\]
for some $z\in B_1(0)$ such that ${\rm dist}(z, L(v_\circ)) \ge \eps$. Since $v_\circ$ is $\eta$-homogeneous, $N(0^+, v_\circ(\,\cdot + z))\ge \eta$ implies that $z$ belongs to the invariant set of $v_\circ$ (see, for instance, \cite[Lemma 5.2]{FS18}). This contradicts ${\rm dist}(z, L(v_\circ)) \ge \eps$, and we are done.
\end{proof}

The following is a very general and standard lemma. We give the proof for completeness.
We thank B. Krummel, from whom we learned this proof.

\begin{lem}
\label{lem.DimLemma}
There exists $\beta:(0,\infty) \rightarrow(0,\infty)$ with $\beta(t) \to 0$ as $t\downarrow 0$, such that the following holds true.

Let $\eps > 0$.  Let $A \subseteq \R^n$ such that for each $y \in A$ and $\rho \in (0,\rho_\circ)$ there exists a $j$-dimensional linear subspace $L_{y,\rho}$ of $\R^n$ for which
\[
A \cap B_{\rho}(y) \subset \{ x: {\rm dist}(x, y+L_{y,\rho}) < \eps\rho \} .
\]
(Note that we do not claim that $L_{y,\rho}$ is unique.)  Then $\mathcal{H}^{j+\beta(\eps)}(A) = 0$.
\end{lem}
\begin{proof}
Let $\beta(t) = n+1-j$ for $t \geq 1/8$ and observe that $\mathcal{H}^{n+1}(A) = 0$.  Thus it suffices to consider $\eps \in (0,1/8)$.

By a covering argument, after rescaling and translating, we may assume that $A \subseteq B_1(0)$ and $0 \in A$.  By assumption, there exists a subspace $L_{0,1}$ such that
\[
A \cap B_1(0) \subset \{ x : {\rm dist}(x, y+L_{0,1}) < \eps \}.
\]
We now cover $L_{0,1}$ by a finite collection of balls $\{B_{2\eps}(z_k)\}_{k = 1,2,\dots,N}$ where $z_k \in L_{0,1}$ for each $k$ and $N \leq C(j) \eps^{-j}$. Observe that $\{B_{2\eps}(z_k)\}_{k = 1,2,\dots,N}$ covers $\{x : {\rm dist}(x, y+L_{0,1}) < \eps \}$ and thus it covers $A \cap B_1(0)$ as well.  Now discard the balls that do not intersect $A$, and for the remaining balls let $y_k \in A \cap B_{2\eps}(z_k)$, so that $\{B_{4\eps}(y_k)\}_{k = 1,2,\dots,N}$ covers $A \cap B_1(0)$, $y_k \in A$, $N \leq C(j) \eps^{-j}$, and $N (4\eps)^{j+\beta} \leq C(j) \eps^{\beta}$.  Choose $\beta =\beta(\eps)$ so that $C(j)\eps^\beta \leq 1/2$.

Now observe that we can repeat this argument with $B_{4\eps}(y_k)$ in place of $B_1(0)$ to get a new covering $\{B_{(4\eps)^2}(y_{k,l})\}_{l=1,2,\dots,N_k}$ of $A \cap B_{4\eps}(y_k)$ with $N_k (4\eps)^{j+\beta} < 1/2$.  Thus $\{B_{(4\eps)^2}(y_{k,l})\}_{k = 1,2,\dots,N, \, l=1,2,\dots,N_k}$ covers $A$ with $y_{k,l} \in A$ and $\sum_{k=1}^N N_k (4\eps)^{2 \cdot (j+\beta)} < (1/2)^2$.  Repeating this argument for a total of $p$ times, we get a finite covering of $A$ by $M$ balls centered on $A$, radii equal to $(4\eps)^p$ (recall $4\eps \le \frac12$), and $M (4\eps)^{p(j+\beta)} < (1/2)^p$.  Thus $\mathcal{H}^{j+\beta}_{(4\eps)^p}(A) \le C(1/2)^p$ for every integer $p = 1,2,3,\dots$.  Letting $p \to \infty$, we get $\mathcal{H}^{j+\beta(\eps)}(A) = 0$.
\end{proof}

Thus, we can directly prove Proposition~\ref{prop.smallother}.

\begin{proof}[Proof of Proposition~\ref{prop.smallother}]
We want to show that ${\Gamma}_*(u)$ has Hausdorff dimension at most $n-2$.  Let $\eps > 0$ and define, for $i \in \N$, $G_i$ to be the set of all points $x_\circ \in {\Gamma}_*(u)$ such that the conclusion of Lemma~\ref{lem.epsclose} holds true with $\delta = 1/i$. In particular, by Lemma~\ref{lem.epsclose} we have ${\Gamma}_*(u) = \bigcup_i G_i$.  For each $q \in \N$, define
\[
G_{i,q} = \{ x_\circ \in G_i : (q-1)/i < N(0^+,  u^{x_\circ})\leq q/i \} .
\]
Observe that ${\Gamma}_*(u) = \bigcup_{i,q} G_{i, q}$, and for every $x_\circ \in G_{i, q}$, $N(0^+, u^{x_\circ})-1/i \le (q-1)/i$ so that we have
\[
G_{i,q} \subset \{ y :  N(0^+,  u^{y}) >  N(0^+,  u^{x_\circ}) - 1/i \}
\]
so that, by  Lemma~\ref{lem.epsclose}, for every $\rho \in (0,1/i]$ there exists a $(n-2)$-dimensional linear subspace $L_{x_\circ,\rho}$ of $\mathbb{R}^n \times \{0\}$ such that
\[
G_{i, q} \cap B_{\rho}(x_\circ) \subset \{ x : {\rm dist}(x, x_\circ+L_{x_\circ,\rho}) < \eps \rho \} .
\]
Now, by Lemma~\ref{lem.DimLemma} with $A = G_{i,q}$ (taking $\rho_\circ = 1/i$ uniform on $G_{i,q}$), $\mathcal{H}^{n-2+\beta(\eps)}(G_{i,q}) = 0$.  Hence $\mathcal{H}^{n-2+\beta(\eps)}({\Gamma}_*(u)) = 0$.  Since $\eps$ is arbitrary, for all $\beta > 0$ we have $\mathcal{H}^{n-2+\beta}({\Gamma}_*(u)) = 0$, and thus ${\Gamma}_*(u)$ has Hausdorff dimension at most $n-2$.

The fact that for $n=2$, ${\Gamma}_*(u)$ is discrete, follows by similar arguments in a standard way.
\end{proof}

\section{$C^\infty$ obstacles}
\label{sec.cinf}

Let us suppose now that the obstacle $\varphi\in C^\infty(B_1')$ is not analytic, and therefore, we cannot reduce to the zero obstacle situation. Our problem is then
\begin{equation}
\label{eq.thinobst_intro_4}
  \left\{ \begin{array}{rcll}
    u& \ge &\varphi & \textrm{ on } B_1\cap \{x_{{n+1}} = 0\}\\
  \Delta u&=&0 & \textrm{ in } B_1\setminus(\{x_{n+1} = 0\}\cap \{u = \varphi\})\\
  \Delta u&\le&0 & \textrm{ in } B_1,
  \end{array}\right.
\end{equation}
where, as before, we are assuming that our solution is even in the $x_{n+1}$-variable.

Let us assume that 0 is a free boundary point, $0\in \partial_{\R^{n}}\{ u = \varphi\}$. Given $\tau \in \N_{\ge 2}$, let us consider the $\tau$-order expansion of $\varphi(x')$ at 0, given by $Q_\tau(x')$. In particular, $(\varphi - Q_\tau)(x') = O(|x'|^{\tau+1})$. Let $Q_\tau^h(x', x_{n+1})$ be the unique even harmonic extension of $Q_\tau$ to $B_1$. Let us now define
\[
\bar u(x', x_{n+1}) := u(x', x_{n+1}) - \varphi(x') + Q_\tau(x') - Q_\tau^h(x', x_{n+1}).
\]

Then, $\bar u$ solves the zero thin obstacle problem with a right-hand side, 
\[
  \left\{ \begin{array}{rcll}
    \bar u& \ge &0 & \textrm{ on } B_1\cap \{x_{{n+1}} = 0\}\\
  \Delta \bar u&=&f & \textrm{ in } B_1\setminus(\{x_{n+1} = 0\}\cap \{u = \varphi\})\\
  -\Delta \bar u&\ge&f & \textrm{ in } B_1,
  \end{array}\right.
\]
where 
\[
f(x) = \Delta_{x'}(Q_\tau(x') - \varphi(x')) = O(|x'|^{\tau -1}). 
\]

Since $|f|\le M |x'|^{\tau-1}$ and $\|\nabla u\|_{L^\infty(B_{1/2})}\le M$ for some constant $M >0$, we can consider the generalized frequency formula, 
\[
\Phi_{\tau}(r, \bar u) := (r + C_M r^2) \frac{d}{dr}\log\max\left\{H(r), r^{n+2\tau}\right\}, \quad\text{where}\quad H(r) := \int_{\partial B_r} \bar u^2,
\]
(cf. \eqref{eq.compwith}) and the constant $C_M$ depends only on the dimension and $M$. Then, there exists some $r_M > 0$ such that  $\Phi_\tau(r, \bar u)$ is non-decreasing for $0 < r < r_M$. In particular, $\Phi_\tau(0^+, \bar u)$ is well-defined and 
\[
n+3\le \Phi_\tau(0^+, \bar u) \le n +2\tau
\]
(see \cite{CSS08, GP09}). We say that the origin is a free boundary point of order $\kappa< \tau$ if $\Phi_\tau(0^+, \bar u) = n+2\kappa$ (in particular, as before, $\kappa \ge \frac32$). If $\kappa = \tau$, we say that the origin is a free boundary point of order \emph{at least} $\tau$. At this point, all the theory developed above for regular free boundary points and singular points, also applies to the situation where there are non-analytic (i.e., non-zero) obstacles, by using the new generalized frequency formula. Notice that this theory can be developed even if the obstacle $\varphi$ has lower regularity than $C^\infty$.  

Finally, we say that the origin is a free boundary point of infinite order if it is of order \emph{at least} $\tau$ for all $\tau > 0$. Notice that this set of free boundary points has not appeared until now, as it did not exist in the zero obstacle case, because in that case the frequency is always finite. 

Intuitively, in the thin obstacle problem \eqref{eq.thinobst_intro_4} a point is of order $\kappa$ when the solution $u$ detaches from the obstacle at order $\kappa$ on the thin space. 

Thus, the free boundary for solutions to the thin obstacle problem with $\varphi\in C^\infty(B_1')$, \eqref{eq.thinobst_intro_4}, can be split as
\[
\Gamma(u) = \Gamma_{3/2}(u) \cup \Gamma_{\rm even}(u) \cup \Gamma_{\rm odd}(u) \cup \Gamma_{\rm half}(u) \cup \Gamma_*(u)\cup \Gamma_\infty(u), 
\]
(cf. \eqref{eq.cffb}), where the new set $\Gamma_\infty(u)$ denotes the set of free boundary points with infinite order. 

The set of points in $\Gamma_\infty(u)$ can be very wild. In fact, the following holds.
\begin{prop}[\cite{FR19}]
\label{prop.Ccontact}
Let $\mathcal{C}\subset B_{1/2}'\subset \R^n$ be any closed set. Then, there exists an an obstacle $\varphi\in C^\infty(B_1')$ and non-trivial solution $u$ to \eqref{eq.thinobst_intro_4} such that $\Lambda(u)\cap B_{1/2} = \{u = \varphi\}\cap B_{1/2} = \mathcal{C}$.
\end{prop}
\begin{proof}
Take any obstacle ${\bar\phi}\in C^\infty(\R^n)$ such that ${\rm supp}\,{\bar\phi}\subset\subset B_{1/8}(\frac34\boldsymbol{e}_1)$, with ${\bar\phi}> 0$ somewhere, and take the non-trivial solution to \eqref{eq.thinobst_intro_4} with obstacle ${\bar\phi}$.

Notice that $ u > {\bar\phi}$ in $B_{1/2}'$ (in particular, $u\in C^\infty(B_{1/2})$). Let $f_{\mathcal{C}}:B_1'\to \R$ be any $C^\infty$ function such that $0\le f_{\mathcal{C}} \le 1$ and $\mathcal{C} = \{f_{\mathcal{C}} = 0\}$.

Now let $\eta\in C^\infty_c(B'_{5/8})$ such that $\eta \ge 0$ and $\eta \equiv 1 $ in $B'_{1/2}$. Consider, as new obstacle, $\varphi = {\bar\phi} + \eta(u-{\bar\phi})(1-f_\mathcal{C}) \in C^\infty(B'_{1/2})$. Notice that $u - \varphi\ge 0$. Notice, also, that for $x'\in B_{1/2}$, $(u - \varphi)(x') =  0$ if and only if $x'\in \mathcal{C}$. Thus, $u$ with obstacle $\varphi$ gives the desired result.
\end{proof}

That is, the contact set can, a priori, be any closed set. In particular, the free boundary can have arbitrary Hausdorff dimension ($n-\eps$ for any $\eps > 0$). It is worth mentioning that the points constructed like this do not exert a force as an obstacle (that is, the Laplacian around them vanishes). 
 
\section{Generic regularity}
\label{sec.11}
We have seen that, in general, the non-regular (or degenerate) part of the free boundary can be of the same size (or even larger, in the case of $C^\infty$ obstacles) than the regular part. This is not completely satisfactory, since we only know how to prove smoothness of the free boundary around regular points. 

It is for this reason that generic regularity results are specially interesting: even if there exist solutions where degenerate points are larger than regular points, we will see that this is not true for a generic solution. That is, for \emph{almost every} solution, the free boundary is smooth up to a lower dimensional set. Let us start by defining what we mean by ``almost every'' solution. 

Let $\varphi\in C^\infty(B_1')$ and let $g\in C^0(\partial B_1)$ even with respect to $x_{n+1}$. Let $\lambda\in [0, 1]$, and let $u_\lambda$ be the solution to 
\begin{equation}
\label{eq.ulambda}
  \left\{ \begin{array}{rcll}
  u_\lambda& \ge &\varphi & \textrm{ on } B_1\cap \{x_{{n+1}} = 0\}\\
  \Delta u_\lambda&=&0 & \textrm{ in } B_1\setminus(\{x_{n+1} = 0\}\cap \{u = \varphi\})\\
  \Delta u_\lambda&\le&0 & \textrm{ in } B_1\\
  u_\lambda & = & g+\lambda & \textrm{ on } \partial B_1. 
  \end{array}\right.
\end{equation}

That is, we consider the set of solutions $\{u_\lambda\}_{\lambda\in [0, 1]}$ with a fixed obstacle $\varphi$ by raising the boundary datum by $\lambda$. Alternatively, we could raise (or lower) the obstacle, or just make small perturbations (monotone) of the boundary value. We say that a property holds for \emph{almost every} solution if it holds for a.e. $\lambda\in [0, 1]$ for any such construction of solutions. 

Now notice that since points of order $\kappa$ are detaching from the obstacle with power $\kappa$, when raising the boundary datum, the larger $\kappa$ is, the faster the free boundary is disappearing (and thus, the less common is that type of point). As a consequence, establishing a quantitative characterization of this fact together with a GMT lemma (coming from \cite{FRS19}), one can show the following proposition. We recall that, given a solution $v$ to a thin obstacle problem \eqref{eq.thinobst_intro_4}, we denote by $\Gamma_{\ge \kappa}(v)$ the set of free boundary points of order (or frequency) greater or equal than $\kappa$. 
 
 \begin{prop}[\cite{FR19}]
 \label{prop.gammakappa}
 Let $\varphi\in C^\infty(B_1')$ and let $g\in C^0(\partial B_1)$ even with respect to $x_{n+1}$. Let $\{u_\lambda\}_{\lambda\in [0,1]}$ the family of solutions to the thin obstacle problem \eqref{eq.ulambda}. Then,
 \begin{itemize}
 \item If $3 \le \kappa \le n+1$, the set $\Gamma_{\ge \kappa} (u_\lambda) $ has Hausdorff dimension at most $n-\kappa + 1$ for almost every $\lambda\in [0, 1]$. 
 \item If $\kappa > n+1$, the set $\Gamma_{\ge \kappa} (u_\lambda) $ is empty for all $\lambda\in [0, 1]\setminus \mathcal{E}_\kappa$, where $\mathcal{E}_\kappa$ has Hausdorff dimension at most $\frac{n}{\kappa-1}$. 
 \item The set $\Gamma_\infty(u_\lambda)$ is empty for all $\lambda\in [0, 1]\setminus \mathcal{E}$, where $\mathcal{E}$ has Minkowski dimension equal to 0.
 \end{itemize}
 \end{prop}
 
That is, we show that, from frequency 3, the higher the frequency, the rarer the points are. This is not unexpected: roughly speaking, around points of frequency $\kappa$ the solution $u$ detaches like a power $\kappa$ from zero. The higher the value of $\kappa$, the less space we have to fit such points for many $\lambda$ (or heights).  
 
 On the other hand, by means of a Monneau-type monotonicity formula one can also show that the set $\bigcup_{\lambda\in [0,1]}\Gamma_{2}(u_\lambda)$ (union of singular points of order $2$ for all $\lambda\in [0,1]$) is contained in a countable union of $(n-1)$-dimensional $C^1$ manifolds. This result is interesting in itself: from Theorem~\ref{thm.gp09} we knew that for each fixed $\lambda_*$, $\Gamma_2(u_{\lambda_*})$ is contained in the union of $(n-1)$-dimensional $C^1$ manifolds. We now claim that, in fact, we can consider all $\Gamma_2(u_\lambda)$ for $\lambda\in (0, 1)$, and this is still contained in the union of $(n-1)$-dimensional $C^1$ manifolds (in some sense, Monneau's monotonicity formula does not care about the $\lambda$ we are at). 
  
 As a consequence, the singular set cannot be too large for too many $\lambda\in (0, 1)$, and we have:
 
  \begin{prop}[\cite{FR19}]
 \label{prop.singpoints}
 Let $\varphi\in C^\infty(B_1')$ and let $g\in C^0(\partial B_1)$ even with respect to $x_{n+1}$. Let $\{u_\lambda\}_{\lambda\in [0,1]}$ the family of solutions to the thin obstacle problem \eqref{eq.ulambda}. 
 
 Then, $\Gamma_{2}(u_\lambda)$ has dimension at most $n-3$ for a.e. $\lambda\in [0, 1]$.  
 \end{prop}
 
 And finally, combining Proposition~\ref{prop.gammakappa}, Proposition~\ref{prop.singpoints}, and Proposition~\ref{prop.smallother}, we get the generic regularity theorem we wanted:
 \begin{thm}[\cite{FR19}]
 Let $\varphi\in C^\infty(B_1')$ and let $g\in C^0(\partial B_1)$ even with respect to $x_{n+1}$.  Let $\{u_\lambda\}_{\lambda\in [0,1]}$ the family of solutions to the thin obstacle problem \eqref{eq.ulambda}. 
 
 Then, the set ${\rm Deg}(u_\lambda)$ has Hausdorff dimension at most $n-2$ for a.e. $\lambda\in [0,1]$. 
 \end{thm}
 
 In particular, the free boundary is smooth up to a lower dimensional set, for almost every solution.
 
 The previous theorem also holds true for obstacles with lower regularity. Namely, in the proof of the result, only $C^{3,1}$ regularity of the obstacle is really used. 
\section{Summary}
\label{sec.12}
Let us finish with a summary of the known results for the solutions to the thin obstacle problem. 

Let $\varphi\in C^\infty(B_1')$ and consider an even solution to the thin obstacle problem, with obstacle $\varphi$, 
\begin{equation}
\label{eq.sol}
  \left\{ \begin{array}{rcll}
   u& \ge &\varphi & \textrm{ on } B_1\cap \{x_{{n+1}} = 0\}\\
  \Delta u&=&0 & \textrm{ in } B_1\setminus(\{x_{n+1} = 0\}\cap \{u = \varphi\})\\
  \Delta u&\le&0 & \textrm{ in } B_1.
  \end{array}\right.
\end{equation}

Then, the solution $u$ is $C^{1,1/2}$ on either side of the obstacle. That is, there exists a constant $C$ depending only on $n$ such that
\[
\|u\|_{C^{1,1/2}(B_{1/2}^+)}+\|u\|_{C^{1,1/2}(B_{1/2}^-)} \le C\left(\|\varphi\|_{C^{1,1}(B_1')} + \|u \|_{L^\infty(B_1)}\right).
\]

Moreover, if we denote $\Lambda(u) := \{u = \varphi\}$ the contact set, the boundary of $\Lambda(u)$ in the relative topology of $\R^n$, $\partial_{\R^n}\Lambda(u)$, is the free boundary, and can be divided into two sets
\[
\Gamma(u) = {\rm Reg}(u) \cup {\rm Deg}(u),
\]
the set of {\em regular points},
\[
{\rm Reg}(u) := \left\{x = (x', 0)\in \Gamma(u) : 0<c r^{3/2}\le \sup_{B_r'(x')} (u - \varphi) \le C r^{3/2},\quad\forall r\in (0, r_\circ)\right\},
\]
and the set of non-regular points or {\em degenerate points}
\[
{\rm Deg}(u) := \left\{x = (x', 0)\in \Gamma(u) : 0\le \sup_{B_r'(x')} (u - \varphi) \le Cr^{2},\quad\forall r\in (0, r_\circ)\right\},
\]

Alternatively, each of the subsets can be defined according to the order of the blow-up (the frequency) at that point. Namely, the set of regular points are those whose blow-up is of order $\frac32$, and the set of degenerate points are those whose blow-up is of order $\kappa$ for some $\kappa \in [2, \infty]$.

The free boundary can be further stratified as
\begin{equation}
\label{eq.FBst}
\Gamma(u) = \Gamma_{3/2} \cup \Gamma_{\rm even} \cup \Gamma_{\rm odd} \cup \Gamma_{\rm half} \cup \Gamma_* \cup \Gamma_{\infty},
\end{equation}
% \Gamma := \Gamma(u) = \Gamma_{3/2} \cup \bigcup_{m\ge 1} \Gamma_{2m}(u)\cup \bigcup_{m\ge 1} \Gamma_{2m+3/2}(u)\cup \bigcup_{m\ge 1} \Gamma_{2m+1}(u)\cup \Gamma_*(u) \cup \Gamma_{\infty}(u).
where:
\begin{itemize}[leftmargin=5.5mm]
\item $ \Gamma_{3/2} = {\rm Reg}(u)$ is the set of regular points. They are an open $(n-1)$-dimensional subset of $\Gamma(u)$, and it is $C^\infty$ (see \cite{ACS08, KPS15, DS16}).

\item $\Gamma_{\rm even} = \bigcup_{m\ge 1} \Gamma_{2m}(u)$ denotes the set of points whose blow-ups have even homogeneity. Equivalently, they can also be characterised as those points of the free boundary where the contact set has zero density, and they are often called singular points. They are contained in the countable union of $C^1$ $(n-1)$-dimensional manifolds; see \cite{GP09}. Generically, however, points in $\Gamma_2(u)$ have dimension at most $n-3$, and points in $\Gamma_{2m}(u)$ have dimension at most $n-2m$ for $m\ge 2$; see \cite{FR19}. 

\item $\Gamma_{\rm odd}= \bigcup_{m\ge 1} \Gamma_{2m+1}(u)$ is, a priori, at most $(n-1)$-dimensional and it is $(n-1)$-rectifiable (see \cite{KW13, FS18,  FRS19}), although it is not known whether it exists. Generically, $\Gamma_{2m+1}(u)$ has dimension at most $n-2m$; see \cite{FR19}. 

\item  $\Gamma_{\rm half} = \bigcup_{m\ge 1} \Gamma_{2m+3/2}(u)$ corresponds to those points with blow-ups of order $\frac72$, $\frac{11}{2}$, etc. They are much less understood than regular points, although in some situations they have a similar behaviour. The set $\Gamma_{\rm half}$ is an $(n-1)$-dimensional subset of the free boundary and it is a $(n-1)$-rectifiable set (see \cite{FS18, KW13, FS19}). Generically, the set $\Gamma_{2m+3/2}(u)$ has dimension at most $n-2m-1/2$. 

\item $\Gamma_*$ is the set of all points with homogeneities $\kappa\in (2, \infty)$, with $\kappa\notin \N$ and $\kappa\notin 2\N-\frac12$. This set has Hausdorff dimension at most $n-2$, so it is always \emph{small}.

\item  $\Gamma_{\infty}$ is the set of points with infinite order (namely, those points at which $u-\varphi$ vanishes at infinite order). For general $C^\infty$ obstacles it could be a huge set, even a fractal set of infinite perimeter with dimension exceeding $n-1$. When $\varphi$ is analytic, instead, $\Gamma_\infty$ is empty. Generically, this set is empty; see \cite{FR19}.
\end{itemize}

\end{document}